\documentclass[reqno,10pt]{amsart}
\usepackage{amsfonts, amsmath, amssymb, amsthm, color}
\allowdisplaybreaks[1]
\newtheorem{thm}{Theorem}[section]
\newtheorem{lemma}{Lemma}[section]
\newtheorem{prop}{Proposition}[section]

\newtheorem{rmk}{Remark}[section]
\theoremstyle{definition}
 \numberwithin{equation}{section}
\newcommand{\rr}{\mathbb{R}}
\newcommand{\al}{\alpha}
\newcommand{\de}{\delta}
 \newcommand{\eps}{\varepsilon}
\newcommand{\la}{\lambda}
\newcommand{\ga}{\gamma}

\newcommand{\calP}{\mathcal P}

\newcommand{\Nrx}{N_{\rho}^\xi}
\newcommand{\Rrx}{R_{\rho}^\xi}
\newcommand{\Wrx}{W_{\rho}^\xi}
\newcommand{\fr}{f_{\rho}}

\newcommand{\Lrx}{\mathcal L_\rho^\xi}
\newcommand{\Trx}{\mathcal T_\rho^\xi}

\renewcommand{\(}{\left(}
\renewcommand{\)}{\right)}
\renewcommand{\[}{\left[}
\renewcommand{\]}{\right]}
\newcommand{\<}{\left\langle}
\renewcommand{\>}{\right\rangle}
\begin{document}
\title[A Liouville equation with variable intensities]{Concentrating solutions for a Liouville type equation
with variable intensities in 2D-turbulence}
\author{Angela Pistoia}
\address[Angela Pistoia] {Dipartimento SBAI, Sapienza Universit\`{a} di Roma, 
Via Antonio Scarpa 16, 00161 Roma, Italy}
\email{pistoia@dmmm.uniroma1.it}
\author{Tonia Ricciardi}
\address[Tonia Ricciardi] {Dipartimento di Matematica e Applicazioni, Universit\`{a} di Napoli Federico II, 
Via Cintia, Monte S.~Angelo, 80126 Napoli, Italy}
\email{tonricci@unina.it}
\thanks{A.P.\ was partially supported by Progetto Sapienza Universit\`{a} di Roma
{\em Nonlinear PDE's: qualitative analysis of solutions.}
T.R.\ was partially supported by PRIN {\em Aspetti variazionali e perturbativi nei problemi differenziali nonlineari}}
\begin{abstract}
We construct sign-changing concentrating solutions for a mean field equation describing turbulent Euler flows with 
variable vortex intensities and arbitrary orientation. 
We study the effect of variable intensities and orientation on the bubbling profile
and on the location of the vortex points. 
\end{abstract}
\date{May 14, 2015}
\maketitle
%%%%%%%%%%%%%%%%%%%%%%%%%%%%%%%%%%%%%%%%%%%%%%%%%%%%%%%%%%%%%%%%%%%%%%%%%%%%%%%
%%%%%%%%%%%%%%%%%%%%%%%%%%%%%%%%%%%%%%%%%%%%%%%%%%%%%%%%%%%%%%%%%%%%%%%%%%%%%%%
\section{Introduction and main results}
\label{sec:intro}
%%%%%%%%%%%%%%%%%%%%%%%%%%%%%%%%%%%%%%%%%%%%%%%%%%%%%%%%%%%%%%%%%%%%%%%%%%%%%%%
%%%%%%%%%%%%%%%%%%%%%%%%%%%%%%%%%%%%%%%%%%%%%%%%%%%%%%%%%%%%%%%%%%%%%%%%%%%%%%%
Motivated by the mean field equation derived by C.~Neri \cite{Neri} in the context
of the statistical mechanics description of 2D-turbulence
within the framework developed by Onsager~\cite{On, EyinkSreenivasan}, Caglioti et al.~\cite{CLMP}, Kiessling~\cite{Kiessling},
we are interested in the existence and in the qualitative properties of solutions to the following problem:
\begin{equation}
\label{p}
\left\{
\begin{aligned}
-\Delta u=&\rho^2\(e^{u}-\tau e^{-\ga u}\)&&\hbox{in}\ \Omega ,\\
   u=&0
&&\hbox{on}\ \partial \Omega ,
\end{aligned}\right.
\end{equation}
where $\rho>0$ is a small constant, $\ga,\tau>0$, and $\Omega\subset\rr^2$ is a smooth bounded domain. 
\par
We recall that the mean field equation for
the $N$-point vortex system with random intensities derived in \cite{Neri} is given by:
\begin{equation}
\label{eq:Neri}
\left\{\begin{aligned}
-\Delta v =&\frac{\int_Ire^{-\beta rv}\,\mathcal P(dr)}{\iint_{I\times\Omega}e^{-\beta r'v}\,\mathcal P(dr') dx}&&\hbox{in}\ \Omega\\
   v =&0&&\hbox{on}\ \partial\Omega.
\end{aligned}\right.
\end{equation}
Here, $v$ is the stream function of a turbulent Euler flow,
$\calP$ is a Borel probability measure on a bounded interval $I$, normalized to $I=[-1,1]$,
describing the vortex intensity distribution, 
and $\beta\in\rr$ is a constant related to the inverse temperature. 
The mean field equation \eqref{eq:Neri} is derived from the classical Kirchhoff-Routh function 
for the $N$-point vortex
system (see, e.g., \cite{BaPi} and the references therein):
\begin{equation}
\label{def:HKR}
H_{KR}^N(r_1,\ldots,r_N, x_1,\ldots,x_N)
=\sum_{i\neq j}r_ir_jG(x_i,x_j)+\sum_{i=1}^Nr_i^2H(x_i,x_i),
\end{equation}
under the ``stochastic" assumption that the $r_i$'s are independent identically distributed
random variables with distribution $\mathcal P$.
Here, $G(x,y)$ denotes the Green's function for the Laplace operator on $\Omega$, namely
\begin{equation}
\label{def:G}
\left\{
\begin{aligned}
-\Delta G(\cdot,y)=&\delta_y&&\hbox{in\ }\Omega\\
G(\cdot,y)=&0&&\hbox{on\ }\partial\Omega
\end{aligned}
\right.
\end{equation}
and
\begin{equation}
\label{def:H}
H(x,y)=G(x,y)+\frac{1}{2\pi}\log|x-y|
\end{equation}
denotes the regular part of $G$.
Setting $u:=-\beta v$ and $\lambda=-\beta$, problem~\eqref{eq:Neri}
takes the form
\begin{equation}
\label{eq:Neribis}
\left\{\begin{aligned}
-\Delta u =&\lambda\frac{\int_Ire^{ru}\,\mathcal P(dr)}{\iint_{I\times\Omega}e^{r'u}\,\mathcal P(dr') dx}&&\hbox{in}\ \Omega\\
   u =&0&&\hbox{on}\ \partial\Omega.
\end{aligned}\right.
\end{equation}
We note that when $\mathcal P(dr)=\delta_1(dr)$ problem~\eqref{eq:Neribis}
reduces to the Liouville type problem
\begin{equation*}
\left\{\begin{aligned}
-\Delta u =&\lambda\frac{e^{u}}{\int_{\Omega}e^{u}\,dx}&&\hbox{in}\ \Omega\\
   u =&0&&\hbox{on}\ \partial\Omega,
\end{aligned}\right.
\end{equation*}
which has been extensively analyzed, see, e.g., \cite{Lin} and the references therein.
On the other hand, when $\mathcal P(dr)=(\delta_1(dr)+\delta_{-1}(dr))/2$, 
problem~\eqref{eq:Neribis} reduces to the sinh-Poisson type problem
\begin{equation}
\label{eq:sinh}
\left\{\begin{aligned}
-\Delta u =&\lambda\,\frac{e^{u}-e^{-u}}{\int_{\Omega}(e^{u}+e^{-u})\,dx}&&\hbox{in}\ \Omega\\
u =&0&&\hbox{on}\ \partial\Omega.
\end{aligned}\right.
\end{equation}
Sign-changing blow-up solutions to problem~\eqref{eq:sinh} were constructed in \cite{BP}.
\par
Here, motivated by the results in \cite{BP}, we are interested in 
identifying some qualitative properties of sign-changing
blowing up solutions to \eqref{eq:Neri} which are specifically related to
\textit{variable intensities and orientations}.
The key features of this situation are captured by taking
$\calP(dr)=\tau_1\de_{1}(dr)+\tau_2\de_{-\ga}(dr)$,
$\ga\in(0,1)$, $0\le\tau_1,\tau_2\le1$, $\tau_1+\tau_2=1$.
Then, problem~\eqref{eq:Neribis}
takes the form
\begin{equation}
\label{eq:Nerispecial}
\left\{\begin{aligned}
-\Delta u =&\lambda\frac{\tau_1e^{u}-\tau_2\ga e^{-\ga u}}{\int_{\Omega}(\tau_1e^{u}+\tau_2e^{-\ga u})dx}&&\hbox{in}\ \Omega\\
   u =&0&&\hbox{on}\ \partial\Omega.
\end{aligned}\right.
\end{equation}
Setting 
\begin{align}
\label{def:def:physconstants}
&\tau:=\frac{\tau_2\ga}{\tau_1},
&&\rho^2:=\frac{\la}{\int_\Omega(e^u+\frac{\tau}{\ga}e^{-\ga u})},
\end{align}
we are reduced to problem~\eqref{p}.
It may be checked that, along a blow-up sequence, we necessarily have $\int_\Omega{e^u}\,dx\to+\infty$,
see \cite{RZ}. Therefore, as far as blow-up solution sequences are concerned, problem~\eqref{eq:Nerispecial}
is equivalent to problem~\eqref{p} with $\rho\to0$.
\par
In this article, we are interested in constructing solution sequences $u=u_\rho$ having a positive 
blow-up point at $\xi_1\in\Omega$ and a negative blow-up point
at $\xi_2\in\Omega$, for some $\xi_1\neq\xi_2$. Moreover, we are interested in the qualitative properties of solutions as 
$\ga$ approaches its limit values.
\par
In order to state our results, let $\mathcal F_2\Omega$ denote the set of pairs of distinct points in $\Omega$,
namely
$$
\mathcal F_2\Omega:=\{(x,y)\in\Omega\times\Omega:x\neq y\}
$$
and let
$\mathrm{cat}(\mathcal F_2\Omega)$ denote the Ljusternik-Schnirelmann category of $\mathcal F_2\Omega$.
\par
We consider the ``Hamiltonian function" $\mathcal H_\gamma:\mathcal F_2\Omega\to\mathbb R$
defined by
\begin{equation}
\label{def:F}
\mathcal H_\gamma(\xi_1,\xi_2)=H(\xi_1,\xi_1)+\frac{H(\xi_2,\xi_2)}{\ga^2}
-\frac{2G(\xi_1,\xi_2)}{\ga}.
\end{equation}
Our first result concerns the existence of sign-changing solutions to \eqref{p}
which are approximately the difference of two Liouville bubbles.
\begin{thm}
\label{thm:main}
There exists $\rho_0>0$ such that for any $\rho\in(0,\rho_0)$
problem~\eqref{p} admits at least $\mathrm{cat}(\mathcal F_2\Omega)$ sign-changing solutions $u_\rho^i$,
$i=1,\ldots,\mathrm{cat}(\mathcal F_2\Omega)$, 
with the property
$$
u_\rho^i(x)\to8\pi G(x,\xi_{1}^i)-\frac{8\pi}{\ga}G(x,\xi_{2}^i)
$$ 
in $C^1_{\mathrm{loc}}\(\Omega\setminus\{\xi_1^i,\xi_2^i\}\)\cap W_0^{1,q}(\Omega)$ for all $q\in[1,2)$
for some critical point $(\xi_{1}^i,\xi_{2}^i)\in\mathcal F_2\Omega$ for $\mathcal H_\ga$.
Moreover,
\begin{enumerate}
  \item [(i)]
The solutions $u_\rho^i$ have the form:
\begin{align*}
u_\rho^i(x)=&\log\Big(\frac{1}{(\delta_{\rho,1}^i)^2+|x-\xi_{\rho,1}^i|^2}\Big)^2+8\pi H(x,\xi_{\rho,1}^i)\\
&\qquad-\frac{1}{\ga}\Big[\log\Big(\frac{1}{(\delta_{\rho,2}^i)^2+|x-\xi_{\rho,2}^i|^2})^2+8\pi H(x,\xi_{\rho,2}^i)\Big]+\phi_\rho^i+O(\rho^2),
\end{align*}
where $\|\phi_\rho^i\|\le C\rho^{2/p}$ for any $p>1$, 
the constants $\delta_{\rho,1}^i,\delta_{\rho,2}^i>0$ are given by
\begin{equation*}
\begin{aligned}
&(\de_{\rho,1}^i)^2=\frac{\rho^2}{8}\exp\left\{8\pi H(\xi_{\rho,1}^i,\xi_{\rho,1}^i)-\frac{8\pi}{\ga}G(\xi_{\rho,1}^i,\xi_{\rho,2}^i)\right\}\\
&(\de_{\rho,2}^i)^2=\frac{\rho^2\tau\ga}{8}\exp\left\{8\pi H(\xi_{\rho,2}^i,\xi_{\rho,2}^i)-8\pi\ga G(\xi_{\rho,1}^i,\xi_{\rho,2}^i)\right\}
\end{aligned}
\end{equation*}
with $(\xi_{\rho,1}^i,\xi_{\rho,2}^i)\in\mathcal F_2\Omega$ satisfying $(\xi_{\rho,1}^i,\xi_{\rho,2}^i)\to(\xi_{1}^i,\xi_{2}^i)$.
\item[(ii)]
The set $\Omega\setminus\{x\in\Omega: u_\rho^i(x)=0\}$ has exactly two connected components.
\item[(iii)]
If $\ga=1$, then \eqref{p} admits $\mathrm{cat}(\mathcal F_2\Omega/(x,y)\sim(y,x))$
\emph{pairs} of solutions $\pm u_\rho^i$ with the above properties. 
\end{enumerate}
\end{thm}
It is not difficult to check (see Lemma~\ref{lem:masses} below) that the solutions $u_\rho^i$ to \eqref{p} 
obtained in Theorem~\ref{thm:main} satisfy
\begin{align}
\label{eq:limitmasses}
&\rho^2\int_\Omega e^{u_\rho^i}\,dx\to 8\pi,
&&\tau\rho^2\int_\Omega e^{-\ga u_\rho^i}\,dx\to \frac{8\pi}{\ga}
\end{align}
as $\rho\to0$, and therefore $u_\rho^i$ yields a solution to \eqref{eq:Nerispecial}
satisfying
$$
\lambda=\rho^2\int_{\Omega}(e^{u_\rho^i}+\frac{\tau}{\ga}e^{-\ga u_\rho^i})\,dx\to8\pi\Big(1+\frac{1}{\ga^2}\Big).
$$
We note that the blow-up mass values obtained in \eqref{eq:limitmasses} are completely determined
by \eqref{p}, see the blow-up analysis contained in Proposition~\ref{prop:blowupanalysis} 
in the Appendix.
\par
We also note that, up to relabelling $(\xi_1,\xi_2)$, the function $\mathcal H_\ga$
defined in \eqref{def:F}
coincides with the Kirchhoff-Routh Hamiltonian \eqref{def:HKR},
as expected.
\par
Our second result, which actually contains the more innovative part of this article, 
is concerned with the asymptotic location of the blow-up points,
in the special case where $\Omega$ is a {\em convex} domain. Roughly speaking, letting $\ga\to+\infty$, the 
``positive bubble" approaches the (unique) maximum point of the Robin's function $H(\xi,\xi)$, whereas the
``negative bubble"
escapes to the boundary $\partial\Omega$, and more precisely
to a point minimizing $\partial_\nu G(x_0,y),$ $y\in\partial\Omega$.
Here $\nu$ denotes the outward normal at the point $y\in\partial\Omega$.
The ``opposite" asymptotic behavior occurs when $\ga\to0^+$.
\begin{thm}
\label{thm:bubbleescape}
Let $\Omega\subset\mathbb R^2$ be a convex bounded domain. 
For every fixed $\ga>0$,
let $u_\rho^\ga$ be a solution sequence to \eqref{p} 
concentrating at $(\xi_{\rho,1}^\ga,\xi_{\rho_2}^\ga)\to(\xi_1^\ga,\xi_2^\ga)$,
as constructed in Theorem~\ref{thm:main}.
We have:
\begin{enumerate}
  \item [(i)]
As $\ga\to+\infty$, we have $\xi_1^\ga\to x_0\in\Omega$, where $x_0$ is the (unique)
maximum point of the Robin function $H(\xi,\xi)$; furthermore,
$\xi_2^\ga\to y_0\in\partial\Omega$, 
where $y_0$ is a minimum point of the function $\partial_\nu G(x_0,y),$ $y\in\partial\Omega$.
\item[(ii)]
Conversely, as $\ga\to0^+$, we have $\xi_1^\ga\to y_0\in\partial\Omega$, 
$\xi_2^\ga\to x_0\in\Omega$, where $x_0,y_0$ are as in part~\textit{(i)}.
\end{enumerate}
\end{thm}
We observe that our method is readily adapted
to yield the existence of one-bubble solutions for the problem:
\begin{equation}
\label{pb:Baraket}
\left\{
\begin{aligned}
-\Delta u=&\rho^2\(e^{u}+\tau e^{\pm\ga u}\)&&\hbox{in}\ \Omega ,\\
u=&0
&&\hbox{on}\ \partial \Omega ,
\end{aligned}\right.
\end{equation}
where $\ga,\tau$ are as above, $\ga\neq1$.
Problem~\eqref{pb:Baraket} was considered in \cite{Baraket2}
(with $\tau=1$ and $\ga\in(0,1)$) in the context of combustion,
where bubbling solutions were constructed by a delicate perturbative method
on the line of \cite{BarPac}. 
\par
Finally, we note that solutions to \eqref{p} also yield solutions to the following related
mean field equation derived in \cite{On}, see also \cite{SawadaSuzuki},
under a ``deterministic" assumption on the vortex intensity distribution:
\begin{equation}
\label{eq:Onsager}
\left\{
\begin{aligned}
-\Delta u=&\lambda\int_I\frac{re^{ru}}{\int_\Omega e^{ru}\,dx}\,\mathcal P(dr)&&\hbox{in\ }\Omega\\
u=&0&&\hbox{on\ }\partial\Omega,
\end{aligned}
\right.
\end{equation}
provided $\mathcal P(dr)=\delta_1(dr)+\mathcal P'$, $\mathrm{supp}\mathcal P'\subset[-r_0,r_0]$ for a suitably small $r_0>0$,
see \cite{ORS,RS,RTZZ}. 
\par
This paper is organized as follows. In Section~\ref{sec:ansatz} we introduce the
notation necessary to the $L^p$-setting of problem~\eqref{p} and we state the Ansatz for 
the sign-changing solutions, following \cite{EGP}.
In Section~\ref{sec:finitedim} we reduce problem~\eqref{p} to a finite dimensional problem on $\mathcal F_2\Omega$.
The equivalent finite dimensional problem is solved in Section~\ref{sec:reducedpb},
thus completing the proof of Theorem~\ref{thm:main}.
In Section~\ref{sec:bubbleescape} we prove Theorem~\ref{thm:bubbleescape}.
The Appendix contains a blow-up analysis for \eqref{p} as well as some technical estimates.
%%%%%%%%%%%%%%%%%%%%%%%%%%%%%%%%%%%%%%%%%%%%%%%%%%%%%%%%%%%%%%%%%%%%%%%%%%%%%%%%%%%%%%%%%%%%%%%%%%%%%%%%%%
%%%%%%%%%%%%%%%%%%%%%%%%%%%%%%%%%%%%%%%%%%%%%%%%%%%%%%%%%%%%%%%%%%%%%%%%%%%%%%%%%%%%%%%%%%%%%%%%%%%%%%%%%%
\section{Ansatz and $L^p-$setting of the problem} 
\label{sec:ansatz}
%%%%%%%%%%%%%%%%%%%%%%%%%%%%%%%%%%%%%%%%%%%%%%%%%%%%%%%%%%%%%%%%%%%%%%%%%%%%%%%%%%%%%%%%%%%%%%%%%%%%%%%%%%
%%%%%%%%%%%%%%%%%%%%%%%%%%%%%%%%%%%%%%%%%%%%%%%%%%%%%%%%%%%%%%%%%%%%%%%%%%%%%%%%%%%%%%%%%%%%%%%%%%%%%%%%%%
Our aim in this section is to
formulate problem \eqref{p} in a more convenient Sobolev space setting,
namely system~\eqref{equ1}--\eqref{equ2} below.
To this end, we first introduce some notation and we recall
some known results.
\par
Henceforth,
$\|u\|_p:=\(\int\limits_{\Omega } |u(x)|^p\,dx\)^{1/p}$
denotes the usual norm  in the Banach space  $L^p(\Omega )$,
$\<u,v\> := \int\limits_{\Omega }\nabla u(x)\cdot\nabla v(x)\,dx $ denotes the usual scalar product in $\mathrm{H}^1_0(\Omega )$
and $\|u\|$  denotes its induced norm on $\mathrm{H}^1_0(\Omega )$.
For any $p>1,$ we denote by $i_p:H^1_0(\Omega)\hookrightarrow L^{p/(p-1)}(\Omega)$
the Sobolev embedding and by $i^*_p:L^{p}(\Omega)\to H^1_0(\Omega)$  the adjoint operator of $i_p$.
That is,  $u=i^*_p(v)$ if and only if  $u\in\mathrm{H}^1_0(\Omega )$ is a weak solution
of $-\Delta u=v$ in $\Omega$.
We point out that $i^*_p$ is a continuous mapping, namely
\begin{equation}
\label{isp}
\|i^*_p(v)\|_{ H^1_0(\Omega)}\le c_p \|v\|_{ L^{p}(\Omega)}, \ \hbox{for any} \ v\in L^{p}(\Omega),
\end{equation}
for some constant $c_p$ which depends on $\Omega$ and $p$.
We define $i^*:\cup_{p>1}L^p(\Omega)\to H_0^1(\Omega)$ by setting
$i^*\vert_{L^p(\Omega)}=i_p^*$ for any $p>1$.
\par
We shall repeatedly use the following well-known inequality \cite{Moe,Tru}.
\begin{lemma}[Moser-Trudinger inequality]
\label{tmt} 
There exists $c>0$ such that for any   bounded domain $\Omega$ in $\rr^2$
there holds
$$
\int\limits_\Omega e^{4\pi u^2/\|u\|^2}dx\le c|\Omega|,
$$
for all $u\in{H}^1_0(\Omega)$.
In particular,  there exists $c>0$ such that for any $q\ge1$
\begin{equation}
\label{ineq:MT}
\| e^{u}\|_{L^q(\Omega)}\le c|\Omega|\exp\left\{\frac{q}{16\pi}\|u\|^2\right\},\qquad
\end{equation}
for all $u\in{H}^1_0(\Omega)$.
\end{lemma}
It follows that for any $p>1$ problem~\eqref{p} is equivalent to 
\begin{equation}
\label{ps}
\left\{
\begin{aligned}
&{u}={i^*_p}\[\rho^2\(e^{u}-\tau e^{-\ga u}\)\],\\ 
&{u}\in H^1_0(\Omega).
\end{aligned}\right.
\end{equation}
In order to further reduce \eqref{ps}, we recall that
the solutions to the Liouville problem
\begin{equation}\label{plim}
-\Delta w=e^w\quad \hbox{in}\quad \rr^2,\qquad
\int\limits_{\rr^2}  e^{w(x)}dx<+\infty,
\end{equation}
are given by the ``Liouville bubbles"
\begin{equation}
\label{walfa}
w_{\de,\xi}(x):=\ln{8\de^2\over\(\de^2+|x-\xi|^2\)^2}\quad
x,\xi\in\rr^2,\ \de>0.
\end{equation}
Moreover, for every $\xi\in\rr^2$, $\de>0$ there actually holds
$$
\int_{\rr^2}e^{w_{\de,\xi}(x)}\,dx=8\pi.
$$
We define the projection  $P:H^1(\Omega)\to H_0^1(\Omega)$
as the weak solution to the problem 
\begin{equation}
\label{pro}
\Delta Pu=\Delta u\quad \hbox{in}\ \Omega ,\qquad  P   u=0\quad \hbox{on}\ \partial\Omega,
\end{equation}
for all $u\in H^1(\Omega)$.
We shall use the following expansion, see Proposition~A.1 in \cite{EGP}.
\begin{lemma}[\cite{EGP}]
\label{pro-exp}
Let $w_{\de,\xi}$ be a Liouville bubble as defined in \eqref{walfa}
with $\xi\in\Omega$ and $\delta\to0$.
Then,
$$
Pw_{\de,\xi}(x)=w_{\de,\xi}(x)-\ln(8\de^2)+8\pi H(x,\xi)+O\(\de^2\)
$$
in $C^0(\overline\Omega)\cap C_{\mathrm{loc}}^2(\Omega)$
and
$$
Pw_{\de,\xi}(x)=8\pi G(x,\xi)+O(\delta^2)
$$
in $C^0(\overline\Omega\setminus\{\xi\})\cap C_{\mathrm{loc}}^2(\Omega\setminus\{\xi\})$.
\end{lemma}
Finally, it will be convenient to set
\begin{equation}
\label{def:f}
f_\rho(t)=\rho^2\(e^{t}-\tau e^{-\ga t}\).
\end{equation}
\par
We seek a solution $u$ to problem~\eqref{p} (or equivalently to problem~\eqref{ps}) whose form is approximately 
the difference of two bubbles.
More precisely, we make the following
\subsection{Ansatz}
The solution $u$ is of the form 
\begin{equation}
\begin{aligned}
\label{ans}
u(x):=&W_\rho^\xi(x)+\phi(x),\\ 
W_\rho^\xi(x):=&Pw_1(x)-\frac{Pw_2(x)}{\ga},\ x\in\Omega.
\end{aligned}
\end{equation}
where we denote $w_i=w_{\de_i,\xi_i}$, $i=1,2$, for some $\de_i>0$ and $\xi_i\in\Omega$
with $\xi_1\neq\xi_2$.
\subsection{Choice of $\de_1,\de_2$}
We observe that $\Wrx$ is an approximate solution only if the quantity $\Delta\Wrx+f_\rho(\Wrx)$
is small. This condition uniquely determines $\delta_1,\de_2$. 
Indeed, in view of Lemma~\ref{pro-exp} we have near $\xi_1$ that
\begin{equation*}
e^{\Wrx(x)}=\frac{\exp\{8\pi H(\xi_1,\xi_1)
-\frac{8\pi}{\ga}G(\xi_1,\xi_2)\}}{8\de_1^2}e^{w_1(x)+O(\de_1^2+\de_2^2+|x-\xi_1|)}.
\end{equation*}
Similarly, near $\xi_2$ we have
\begin{equation*}
e^{-\ga\Wrx(x)}=\frac{\exp\{8\pi H(\xi_2,\xi_2)
-8\pi\ga\,G(\xi_1,\xi_2)\}}{8\de_2^2}e^{w_2(x)+O(\de_1^2+\de_2^2+|x-\xi_2|)}.
\end{equation*}
It follows that if the quantity:
\begin{equation}
\label{Wapprox}
R_\rho^\xi:=\Delta\Wrx+f_\rho(\Wrx)=-e^{w_1}+\frac{e^{w_2}}{\ga}
-\rho^2\(e^{\Wrx}-\tau e^{-\ga\Wrx}\)
\end{equation}
is in some sense small, then necessarily $\delta_1,\delta_2$ are given by
\begin{equation}
\label{eq:deltavalues}
\begin{aligned}
&\de_1^2=\frac{\rho^2}{8}\exp\left\{{8\pi H(\xi_1,\xi_1)-\frac{8\pi}{\ga}G(\xi_1,\xi_2)}\right\}\\
&\de_2^2=\frac{\rho^2\tau\ga}{8}\exp\left\{{8\pi H(\xi_2,\xi_2)-8\pi\ga\,G(\xi_1,\xi_2)}\right\}.
\end{aligned}
\end{equation}
Henceforth, we assume \eqref{eq:deltavalues}. We note that in particular
$\delta_1,\de_2$ have the same decay rate as $\rho$.
The precise decay rate of \eqref{Wapprox} is provided in the following lemma and will be used
repeatedly throughout this paper.
\begin{lemma}
\label{lem:deltachoice}
Let $\de_1,\de_2$ be defined by \eqref{eq:deltavalues}.
Then, for all $1\le p<2$ we have
\begin{equation*}
\left\|\rho^2e^{W_\rho^\xi}-e^{w_1}\right\|_{L^p(\Omega)}^p
+\left\|\rho^2\tau\ga e^{-\ga W_\rho^\xi}-e^{w_2}\right\|_{L^p(\Omega)}^p
\le C\rho^{2-p}.
\end{equation*}
In particular, 
\begin{equation}
\label{eq:f}
\|\Rrx\|_{L^p(\Omega)}^p=\|\Delta\Wrx+f_\rho(\Wrx)\|_{L^p(\Omega)}^p\le C\rho^{2-p}
\end{equation}
and
\begin{equation}
\label{eq:fprime}
\|f_\rho'(W_\rho^\xi)-(e^{w_1}+e^{w_2})\|_{L^p(\Omega)}^p\le C\rho^{2-p}.
\end{equation}
\end{lemma}
\begin{proof}
The proof is analogous to the proof of Lemma~B.1 in \cite{EGP}.
Since the asserted estimates are a key point in the $L^p$-setting
of problem~\eqref{p}, we outline the proof for the reader's convenience. 
We need to estimate the quantity
\begin{equation*}
\int_\Omega|\rho^2e^{\Wrx}-e^{w_1}|^p\,dx
=\int_\Omega|\rho^2e^{Pw_1-\ga^{-1}Pw_2}-e^{w_1}|^p\,dx.
\end{equation*}
Recalling the expansions in Lemma~\ref{pro-exp} and the value of $\delta_1$
as in \eqref{eq:deltavalues}, we have:
\begin{align*}
\int_{B_\eps(\xi_1)}&|\rho^2e^{\Wrx}-e^{w_1}|^p\,dx\\
=&\int_{B_\eps(\xi_1)}\left|\rho^2
\exp\Big\{w_1-\log(8\delta_1^2)
+8\pi H(x,\xi_1)-\frac{8\pi}{\ga}G(x,\xi_2)+O(\delta_1^2)\Big\}-e^{w_1}\right|^p\,dx\\
=&\int_{B_\eps(\xi_1)}\left|\rho^2
\exp\Big\{w_1-\log(8\delta_1^2)
+8\pi H(\xi_1,\xi_1)-\frac{8\pi}{\ga}G(\xi_1,\xi_2)+O(\delta_1^2+|x-\xi_1|)\Big\}-e^{w_1}\right|^p\,dx\\
=&\int_{B_\eps(\xi_1)}\left|e^{w_1+O(\delta_1^2+|x-\xi_1|)}-e^{w_1}\right|^p\,dx\\
\le&C\int_{B_\eps(\xi_1)}e^{pw_1}(\delta_1^2+|x-\xi_1|)^p\,dx.
\end{align*}
In turn, using the explicit form of $w_1$, we derive:
\begin{align*}
\int_{B_\eps(\xi_1)}&|\rho^2e^{\Wrx}-e^{w_1}|^p\,dx
\le C\de_1^{2p}\int_{B_\eps(\xi_1)}\frac{(\de_1^2+|x-\xi_1|)^p}{(\de_1^2+|x-\xi_2|^2)^{2p}}\,dx\\
\le&C\de_1^{2-p}\int_{B_{\eps/\de_1}(0)}\frac{(\de_1+|y|)^p}{(1+|y|^2)^{2p}}\,dy
\le C\rho^{2-p}.
\end{align*}
On the other hand, since in $\Omega\setminus B_\eps(\xi_1)$ we have $e^{w_1}\le C\delta_1^2$ 
and $\Wrx\le C$ for some $C>0$ independent of $\rho>0$,
 we readily obtain
\begin{align*}
\int_{\Omega\setminus B_\eps(\xi_1)}\left|\rho^2e^{\Wrx}\right|^p\,dx
+\int_{\Omega\setminus B_\eps(\xi_1)}e^{pw_1}\,dx\le C\rho^{2p}.
\end{align*}
Hence, we conclude that
\begin{equation*}
\left\|\rho^2e^{W_\rho^\xi}-e^{w_1}\right\|_{L^p(\Omega)}^p
\le C\rho^{2-p}.
\end{equation*}
The second decay estimate is obtained similarly.
\end{proof}
Estimate~\eqref{eq:fprime} will be used to prove the key invertibility
estimate for the linearized operator.
\subsection{Condition on $\xi_1,\xi_2$}
The concentration points $\xi_1,\xi_2$ are taken inside $\Omega$, 
far from the boundary of $\Omega$ and distinct, uniformly with respect to $\rho$. 
More precisely, $\xi_1,\xi_2$ satisfy the following condition:
\begin{equation}
\label{xi}
d(\xi_1,\partial\Omega),\ d(\xi_2,\partial\Omega), |\xi_1-\xi_2|\ge \eta\ \hbox{for some}\ \eta>0. 
\end{equation}
\subsection{The error term $\phi$}
The error term $\phi$ belongs to the subspace $K^\perp\subset H_0^1(\Omega)$ which we now define.
It is well known that for every $\de>0$, $\xi\in\rr^2$, the linearized problem
\begin{equation}
\label{lin}
-\Delta \psi=  e^{w_{\de,\xi}}\psi\quad \hbox{in}\quad \rr^2
\end{equation}
has a $3-$dimensional space of bounded solutions generated by the functions
\begin{equation}
\label{def:psi}
\begin{split}
\psi_{\delta,\xi}^j(x):=&\frac{1}{4}\frac{\partial w_{\de,\xi}}{\partial\xi_i}={x_i-\xi_i\over \de^2+|x-\xi|^2},\qquad j=1,2,\\
\psi_{\delta,\xi}^0(x):=&-\frac{\delta}{2}\frac{\partial w_{\de,\xi}}{\partial\de}={\de^2-|x-\xi|^2\over \de^2+|x-\xi|^2}.
\end{split}
\end{equation}
We shall need the following ``orthogonality relations" from \cite{EGP}, Lemma~A.4:
\begin{equation}
\label{eq:psiorthog}
\begin{aligned}
&\|P\psi_{\de,\xi}^0\|^2=\frac{D_0\de}{\rho}[1+O(\rho^2)]\\
&(P\psi_{\de,\xi}^j,P\psi_{\de,\xi}^l)_{H_0^1(\Omega)}=\frac{D\de^2}{\rho^4}[\de_{jl}+O(\rho^2)]\\
&(P\psi_{\de_1,\xi_1}^j,P\psi_{\de_2,\xi_2}^l)_{H_0^1(\Omega)}=O(1)
\end{aligned}
\end{equation}
as $\rho\to0$, uniformly in $\xi,\xi_1,\xi_2$ satisfying $\mathrm{dist}(\xi,\partial\Omega)\ge\eta$
and \eqref{xi}.
Here $D_0=64\int_{\rr^2}(1-|y|^2)/(1+|y|^2)^4$, $D=64\int_{\rr}|y|^2/(1+|y|^2)^4$
and $\delta_{jl}$ denotes the Kronecker symbol.
We set
\begin{align*}
&K:=\textrm{span}\left\{P\psi^j_{\de_i,\xi_i},\ i,j=1,2\right\},
&&K^\perp:=\left\{\phi\in H^1_0(\Omega)\ :\ \<\phi,P\psi^j_{\de_i,\xi_i}\>=0\ i,j=1,2\right\}.
\end{align*}
We also denote by
\begin{align*}     
&\Pi : H^1_0(\Omega)\to{K },      
&&{\Pi ^\perp}: H^1_0(\Omega)\to{K ^\perp }
\end{align*}
the corresponding projections.
Then, problem~\eqref{ps} is reduced to
the following system:
\begin{align}
\label{equ1}
{\Pi^\perp}\[u-{i^*}\(f_\rho(u)\)\]=&0\\
\label{equ2}
{\Pi}\[u-{i^*}\( f_\rho(u)\)\]=&0
\end{align}
where $u$ satisfies Ansatz~\eqref{ans} and $f_\rho$ is defined in \eqref{def:f}.
%%%%%%%%%%%%%%%%%%%%%%%%%%%%%%%%%%%%%%%%%%%%%%%%%%%%%%%%%%%%%%%%%%%%%%%%%%%%%%%%%%%%%%%%%%%%%%
%%%%%%%%%%%%%%%%%%%%%%%%%%%%%%%%%%%%%%%%%%%%%%%%%%%%%%%%%%%%%%%%%%%%%%%%%%%%%%%%%%%%%%%%%%%%%%
\section{The finite dimensional reduction}
\label{sec:finitedim}
%%%%%%%%%%%%%%%%%%%%%%%%%%%%%%%%%%%%%%%%%%%%%%%%%%%%%%%%%%%%%%%%%%%%%%%%%%%%%%%%%%%%%%%%%%%%%%
%%%%%%%%%%%%%%%%%%%%%%%%%%%%%%%%%%%%%%%%%%%%%%%%%%%%%%%%%%%%%%%%%%%%%%%%%%%%%%%%%%%%%%%%%%%%%%
In this section we obtain a solution for equation~\eqref{equ1} for any fixed $\xi_1,\xi_2\in\Omega$ satisfying
\eqref{xi}.
Namely, our aim is to show the following.
\begin{prop}
\label{prop:phi}
For any $p\in(1,2)$ there exists $\rho_0>0$ such that for any $\rho\in(0,\rho_0)$
and for any $\xi_1,\xi_2\in\Omega$ satisfying \eqref{xi} 
there exists a unique $\phi\in{K^\perp}$ such that
equation \eqref{equ1} is satisfied. Moreover,
\begin{equation}\label{fi1}
\|\phi\|=O\(\rho^{(2-p)/p}|\log\rho|\)\ \hbox{uniformly with respect to $\xi$ in compact sets of $\Omega.$}
\end{equation}
\end{prop}
\begin{rmk}
We note that if $(\xi_1,\xi_2)$ is a critical point for $\mathcal H_\ga$, then we actually
have $\|\phi\|=O(\rho^2)$, see Lemma~\ref{lem:phiimproved} in the Appendix.
\end{rmk}
We split the proof into several steps.
\subsection{The linear theory}
We consider the linearized operator  $\mathcal L_\rho^\xi:K^\perp\to K^\perp$ defined by
\begin{equation}
\label{def:linop}
\mathcal L_\rho^\xi\phi=\Pi^\perp\{\phi-i^*[f_\rho'(W_\rho^\xi)\phi]\}
\end{equation}
The following estimate holds.
\begin{prop}
\label{prop:Linvert}
There exists $c_\xi>0$ independent of $\rho$ such that
$$
\|\mathcal L_\rho^\xi\phi\|\ge\frac{c_\xi}{|\log\rho|}\|\phi\|,
\qquad\hbox{for all}\ \phi\in K^\perp.
$$
\end{prop}
The proof of Proposition~\ref{prop:Linvert} may be derived by adapting step-by-step
to our situation the proof of Proposition~3.1 in \cite{EGP}. Here, alternatively,
we choose to prove Proposition~\ref{prop:Linvert} by 
reducing $\mathcal L_\rho^\xi$ to a suitable operator $L_\rho^\xi$ 
to which Proposition~3.1 in \cite{EGP} may be applied directly.
To this end, we first show the following.
\begin{lemma}
\label{lem:EGPinvert}
Let $\xi=(\xi_1,\xi_2)\in\Omega^2$, $\xi_1\neq\xi_2$,
and let $\de_1,\de_2>0$ be such that $0<a\rho\le\de_i\le b\rho$, $i=1,2$
for some $0<a\le b$.
Let $L_\rho^\xi:K^\perp\to K^\perp$ be defined by
$$
L_\rho^\xi\phi=\Pi^\perp\{\phi-i^*[(e^{w_1}+e^{w_2})\phi]\}.
$$
There exists $\tilde c_\xi>0$ depending on $\mathrm{dist}(\xi_i,\partial\Omega)$
and $a,b$ only, such that
$$
\|L_\rho^\xi\phi\|\ge\frac{\tilde c_\xi}{|\log\rho|}\|\phi\|.
$$
\end{lemma}
\begin{proof}
Let $V(x)$ be any smooth positive function defined on $\Omega$
satisfying
$$
\frac{\delta_1}{\rho}=\sqrt{\frac{V(\xi_1)}{8}}e^{4\pi(H(\xi_1,\xi_1)+G(\xi_1,\xi_2))},
\qquad
\frac{\delta_2}{\rho}=\sqrt{\frac{V(\xi_2)}{8}}e^{4\pi(H(\xi_2,\xi_2)+G(\xi_1,\xi_2))},
$$
see formula (2.6) in \cite{EGP}.
That is,
$$
V(\xi_1)=\frac{8\de_1^2}{\rho^2}e^{-8\pi H(\xi_1,\xi_1)-8\pi G(\xi_1,\xi_2)},
\qquad
V(\xi_2)=\frac{8\de_2^2}{\rho^2}e^{-8\pi H(\xi_2,\xi_2)-8\pi G(\xi_1,\xi_2)}.
$$
Then, following Lemma~B.1 in \cite{EGP} (or the proof of Lemma~\ref{lem:deltachoice}), we have
\begin{equation}
\label{est:EGPlemB1}
\|\rho^2V(x)e^{Pw_1+Pw_2}-(e^{w_1}+e^{w_2})\|_{L^p(\Omega)}^p\le C\rho^{2-p}.
\end{equation}
We define the operator $\mathcal L_V\phi$ by setting
$$
\mathcal L_V\phi=\Pi^\perp\{\phi-i^*[\rho^2V(x)e^{Pw_1+Pw_2}\phi]\}.
$$
Then Proposition~3.1 in \cite{EGP} states that there exists $c_V>0$ such that
\begin{equation}
\label{est:LVinvert}
\|\mathcal L_V\phi\|\ge\frac{c_V}{|\log\rho|}\|\phi\|
\end{equation}
for all $\phi\in K^\perp$.
Thus, we may write
\begin{align*}
\|L_\rho^\xi\phi\|=&\|\Pi^\perp\{\phi-i^*[(e^{w_1}+e^{w_2})\phi]\}\|\\
\ge\|\Pi^\perp&\{\phi-i^*[\rho^2V(x)e^{Pw_1+Pw_2}\phi]\}\|
-\|\Pi^\perp\{i^*[(e^{w_1}+e^{w_2}-\rho^2V(x)e^{Pw_1+Pw_2})\phi]\}\|.
\end{align*}
We estimate the last term, for any $1<q<p<2$, using \eqref{est:EGPlemB1}:
\begin{align*}
\|\Pi^\perp\{i^*[(e^{w_1}+e^{w_2}-\rho^2V(x)&e^{Pw_1+Pw_2})\phi]\}\|
\le\|i^*[(e^{w_1}+e^{w_2}-\rho^2V(x)e^{Pw_1+Pw_2})\phi]\|\\
\le&c_q\|(e^{w_1}+e^{w_2}-\rho^2V(x)e^{Pw_1+Pw_2})\phi\|_q\\
\le&c_q\|e^{w_1}+e^{w_2}-\rho^2V(x)e^{Pw_1+Pw_2}\|_p\|\phi\|_{pq/(p-q)}\\
\le&C\rho^{(2-p)/p}\|\phi\|.
\end{align*}
It follows that 
$$
\|L_\rho^\xi\phi\|\ge\|\mathcal L_V\phi\|-C\rho^{(2-p)/p}\|\phi\|
\ge\left(\frac{c_V}{|\log\rho|}-C\rho^{(2-p)/p}\right)\|\phi\|.
$$
Hence, the asseted estimate holds with $\tilde c_\xi=c_V/2$.
\end{proof}
Now the proof of Proposition~\ref{prop:Linvert} is readily derived from
Lemma~\ref{lem:deltachoice} and Lemma~\ref{lem:EGPinvert}.
\begin{proof}[Proof of Proposition~\ref{prop:Linvert}]
We estimate
\begin{align*}
\|\mathcal L_\rho^\xi\phi\|\ge&\|\Pi^\perp\{\phi-i^*[(e^{w_1}+e^{w_2})\phi]\}\|
-\|\Pi^\perp i^*\{[f_\rho'(W_\rho^\xi)-(e^{w_1}+e^{w_2})]\phi\}\|\\
=&\|L_\rho^\xi\phi\|-\|\Pi^\perp i^*\{[f_\rho'(W_\rho^\xi)-(e^{w_1}+e^{w_2})]\phi\}\|.
\end{align*}
On the other hand, for any $q\in[1,p)$, we have
\begin{align*}
\|\Pi^\perp i^*&\{[f_\rho'(W_\rho^\xi)-(e^{w_1}+e^{w_2})]\phi\}\|
\le\|i^*\{[f_\rho'(W_\rho^\xi)-(e^{w_1}+e^{w_2})]\phi\}\|\\
\le&c_q\|[f_\rho'(W_\rho^\xi)-(e^{w_1}+e^{w_2})]\phi\|_q
\le c_q\|f_\rho'(W_\rho^\xi)-(e^{w_1}+e^{w_2})\|_p\|\phi\|_{pq/(p-q)}\\
\le& C\rho^{(2-p)/p}\|\phi\|.
\end{align*}
It follows that for sufficiently small $\rho$ we have 
\begin{align*}
\|\mathcal L_\rho^\xi\phi\|
\ge&\|L_\rho^\xi\phi\|-\|\Pi^\perp i^*\{(f_\rho'(W_\rho^\xi)-(e^{w_1}+e^{w_2}))\phi\}\|\\
\ge&\frac{\tilde c_\xi}{|\log\rho|}\|\phi\|-C\rho^{(2-p)/p}\|\phi\|\ge\frac{c_\xi}{|\log\rho|}\|\phi\|,
\end{align*}
with $c_\xi=\tilde c_\xi/2$.
\end{proof}
%%%%%%%%%%%%%%%%%%%%%%%%%%%%%%%%%%%%%%%%%%%%%%%%%%%%%%%%%%%%%%%%%%%%%%%%%%%%%
\subsection{The contraction argument}
Recall that we seek solutions to system~\eqref{equ1}--\eqref{equ2}
satisfying Ansatz~\eqref{ans}.
Hence, we rewrite equation~\eqref{equ1}:
\begin{equation}
\label{eq:equ1ans}
\Pi^\perp\{\phi-i^*[\fr(\Wrx+\phi)+\Delta\Wrx]\}=0.
\end{equation}
We recall from \eqref{Wapprox} that
$$
R_\rho^\xi=\Delta\Wrx+f_\rho(\Wrx).
$$
Setting 
\begin{align}
\label{def:NR}
\Nrx(\phi):=\fr(\Wrx+\phi)-\fr(\Wrx)-\fr'(\Wrx)\phi,
\end{align}
we may write
\begin{equation}
\label{eq:NR}
\fr(\Wrx+\phi)+\Delta\Wrx=\Nrx(\phi)+\Rrx+\fr'(\Wrx)\phi.
\end{equation}
Hence, using \eqref{eq:NR} and the definition \eqref{def:linop} of $\mathcal L_\rho^\xi$
we may rewrite \eqref{eq:equ1ans} in the form
\begin{equation}
\mathcal L_\rho^\xi\phi=\Pi^\perp\{\phi-i^*[\fr'(\Wrx)\phi]\}=\Pi^\perp\circ i_p^*[\Nrx(\phi)+\Rrx].
\end{equation}
Finally, setting
$$
\Trx(\phi):=(\Lrx)^{-1}\circ\Pi^\perp\circ i^*[\Nrx(\phi)+\Rrx],
$$
we are finally reduced to solve the following fixed point equation for $\phi$:
\begin{equation}
\label{prob:phi}
\phi=\Trx(\phi).
\end{equation}
The existence of a solution for \eqref{prob:phi} will follow from the following.
\begin{prop}
\label{prop:contraction}
For any $p\in(1,2),$ there exists $R_0=R_0(\xi,p)>0$ such that $\Trx$
is a contraction in $\mathcal B_{R_0\rho^{(2-p)/p}|\log\rho|}\subset K^\perp$.   
\end{prop}
\begin{rmk}
We note that Proposition~\ref{prop:contraction} slightly improves 
Proposition~4.1 in \cite{EGP}, where
the condition $p\in(1,4/3)$ is required. 
\end{rmk}
We begin by some lemmas.
The following elementary result is useful to estimate $\Nrx$.
\begin{lemma}
\label{lem:MVT}
Let $\fr\in C^2(\rr,\rr)$ and let $\Nrx$ be correspondingly defined by \eqref{def:NR}.
Then, for all $\phi,\psi\in\rr$ there exist $\eta,\theta\in[0,1]$
such that
$$
|\Nrx(\phi)-\Nrx(\psi)|\le|\fr''(\Wrx+\eta(\theta\phi+(1-\theta)\psi))|\,(|\phi|+|\psi|)\,|\phi-\psi|.
$$
\end{lemma}
\begin{proof}
Applying the Mean Value Theorem twice, we have:
\begin{equation}
\begin{aligned}
\Nrx(\phi)-\Nrx(\psi)=&\fr(\Wrx+\phi)-\fr(\Wrx+\psi)-\fr'(\Wrx)(\phi-\psi)\\
=&\fr'(\Wrx+\psi+\theta(\phi-\psi))(\phi-\psi)-\fr'(\Wrx)(\phi-\psi)\\
=&\fr''(\Wrx+\eta[\theta\phi+(1-\theta)\psi])[\theta\phi+(1-\theta)\psi](\phi-\psi).
\end{aligned}
\end{equation}
The asserted estimate now easily follows.
\end{proof}
\begin{lemma}
\label{lem:Nqeps}
Let $\fr$ be given by \eqref{def:f}.
Let $q>1$, $\eps>0$.
There exists $C>0$ independent of $\al_i,\phi,\rho$ such that
\begin{equation*}
\|\Nrx(\phi)\|_{L^q(\Omega)}\le C\rho^{2(1-q)/q-\eps}
\exp\left\{\frac{1}{4\pi\eps}\|\phi\|^2\right\}\|\phi\|^2.
\end{equation*}
\end{lemma}
\begin{proof}
In view of Lemma~\ref{lem:MVT} with $\psi=0$, and observing that $\Nrx(0)=0$, we have
$$
|\Nrx(\phi)|\le|\fr''(\Wrx+\eta\theta\phi)|\,|\phi|^2
$$
for some $0\le\eta(x),\theta(x)\le1$.
Let $r,s>1$ satisfy $r^{-1}+s^{-1}=q^{-1}$.
Then, by H\"older's inequality and the Moser-Trudinger embedding~\eqref{ineq:MT},
\begin{equation*}
\|\Nrx(\phi)\|_q\le\|\fr''(\Wrx+\eta\theta\phi)\|_r\|\phi^2\|_s
\le C\|\fr''(\Wrx+\eta\theta\phi)\|_r\|\phi\|^2.
\end{equation*}
Since
\begin{equation}
\begin{aligned}
\fr(t)=&\rho^2\(e^{t}-\tau e^{-\ga t}\)\\
\fr'(t)=&\rho^2\(e^{t}+{\ga\tau}e^{-\ga t}\)\\
\fr''(t)=&\rho^2\(e^{t}-\ga^2\tau e^{-\ga t}\)
\end{aligned}
\end{equation}
we have
$$
|\fr''(\Wrx+\eta\theta\phi)|\le\rho^2\left(e^{\Wrx+\eta\theta\phi}
+\tau\ga^2e^{-\ga\Wrx-\ga\eta\theta\phi}\right).
$$
We estimate term by term.
In view of Lemma~\ref{lem:deltachoice}, we have
$$
\left\|{\rho^2}e^{\Wrx}-e^{w_1}\right\|_{L^p(\Omega)}^p\le C\rho^{2-p}.
$$
Let $t,v>1$ be such that $t^{-1}+v^{-1}=r^{-1}$.
In view of H\"older's inequality, the Moser-Trudinger embedding and Lemma~\ref{lem:deltachoice}, we have
\begin{equation*}
\begin{aligned}
\left\|{\rho^2}e^{\Wrx+\eta\theta\phi}\right\|_{L^r(\Omega)}
\le&\left\|{\rho^2}e^{\Wrx}\right\|_{L^t(\Omega)}\|e^{\eta\theta\phi}\|_{L^v(\Omega)}\\
\le&\(\|e^{w_1}\|_{L^t(\Omega)}+\left\|{\rho^2}e^{\Wrx}-e^{w_1}\right\|_{L^t(\Omega)}\)
\|e^{|\phi|}\|_{L^v(\Omega)}\\
\le&C(C_1\rho^{2(1-t)/t}+C_2\rho^{(2-t)/t})\exp\{{v\over16\pi}\|\phi\|^2\}\\
\le&C\rho^{2(1-t)/t}\exp\{{v\over16\pi}\|\phi\|^2\}.
\end{aligned}
\end{equation*}
Similarly, we estimate
\begin{equation}
\left\|{\rho^2\ga^2\tau}e^{-\ga\Wrx-\eta\theta\ga\phi}\right\|_{L^r(\Omega)}
\le C\rho^{2(1-t)/t}\exp\{{\ga^2v\over16\pi}\|\phi\|^2\}.
\end{equation}
We conclude from the above that
$$
\|f_\rho''(\Wrx+\eta\theta\phi)\|_{L^r(\Omega)}
\le C\rho^{2(1-t)/t}\exp\Big\{\frac{v}{16\pi}\|\phi\|^2\Big\},
$$
for every $t,v>1$ such that $t^{-1}+v^{-1}=r^{-1}$.
Choosing $s=v=4/\eps$, we obtain
$r^{-1}=q^{-1}-\eps/4$, $t^{-1}=r^{-1}-\eps/4=q^{-1}-\eps/2$
and consequently
\begin{equation*}
\begin{aligned}
&\frac{2(1-t)}{t}=\frac{2(1-q)}{q}-\eps,
&&\frac{2-t}{t}=\frac{2-q}{q}-\eps.
\end{aligned}
\end{equation*}
The asserted estimate is established.
\end{proof}
\begin{proof}[Proof of Proposition~\ref{prop:contraction}]
Recalling the definition of $\Trx$, we estimate:
\begin{equation*}
\begin{aligned}
\|\Trx(\phi)\|\le&\|(\Lrx)^{-1}\|\|\Pi\circ i^*(\Nrx(\phi)+\Rrx)\|
\le C_{\Lrx}|\log\rho|\(\|i^*(\Nrx(\phi))\|+\|i^*(\Rrx)\|\)\\
\le&C_{\Lrx}|\log\rho|\(c_q\|\Nrx(\phi)\|_q+c_p\|\Rrx\|_p\),
\end{aligned}
\end{equation*}
for any $p,q>1$.
It follows that
\begin{equation*}
\begin{aligned}
\|\Trx(\phi)\|&\le C_{\Lrx}|\log\rho|\times\\
&\quad\times\[\(C_1\rho^{2(1-q)/q-\eps}+C_2\rho^{(2-q)/q-\eps}\)
\exp\{\frac{\|\phi\|^2}{4\pi\eps}\}\|\phi\|^2
+C_3\rho^{(2-p)/p}\].
\end{aligned}
\end{equation*}
Consequently, if $\|\phi\|\le R_0|\log\rho|\rho^{(2-p)/p}$,
we have
\begin{equation*}
\begin{aligned}
\|\Trx(\phi)\|\le C_{\Lrx}|\log\rho|\rho^{(2-p)/p}
\[CR_0^2|\log\rho|^2\rho^{2(1-q)/q-\eps+(2-p)/p}+C_3\].
\end{aligned}
\end{equation*}
For any fixed $p\in(1,2)$ we may find $q>1$ and $\eps>0$
such that 
$$2(1-q)/q-\eps+(2-p)/p>0.
$$
Taking $R_0\ge 2C_{\Lrx}C_3$, we obtain
for sufficiently small $\rho$ that
\begin{equation}
\label{eq:TintoB}
\Trx(\mathcal B_{R_0|\log\rho|\rho^{(2-p)/p}})\subset\mathcal B_{R_0|\log\rho|\rho^{(2-p)/p}}.
\end{equation}
We are left to show that $\Trx$ is a contraction.
We have
\begin{align*}
\|\Trx(\phi)-\Trx(\psi)\|\le C_{\Lrx}|\log\rho|\|i^*[\Nrx(\phi)-\Nrx(\psi)]\|
\le C_{\Lrx}|\log\rho|\,c_q\|\Nrx(\phi)-\Nrx(\psi)\|_q.
\end{align*}
Recalling that
\begin{align*}
\fr''(\Wrx+&\eta[\theta\phi+(1-\theta)\psi])\\
=&\rho^2e^{\Wrx+\eta[\theta\phi+(1-\theta)\psi]}
-\rho^2\tau\ga^2e^{-\ga\Wrx-\ga\eta[\theta\phi+(1-\theta)\psi]}
\end{align*}
we estimate, similarly as above, for any $r,s>1$ such that $r^{-1}+(2s)^{-1}=q^{-1}$
\begin{align*}
\|\Nrx(\phi)&-\Nrx(\psi)\|_q\le\|\fr''(\Wrx+\eta[\theta\phi+(1-\theta)\psi])(|\phi|+|\psi|)(|\phi-\psi|)\|_q\\
\le&\{\|\rho^2e^{\Wrx+\eta[\theta\phi+(1-\theta)\psi]}\|_r
+\|\rho^2\tau\ga^2e^{-\ga\Wrx-\ga\eta[\theta\phi+(1-\theta)\psi]}\|_r\}\times\\
&\qquad\qquad\times(\|\phi\|_s+\|\psi\|_s)(\|\phi-\psi\|_s)\\
\le&C\{\|\rho^2e^{\Wrx+\eta[\theta\phi+(1-\theta)\psi]}\|_r
+\|\rho^2\tau\ga^2e^{-\ga\Wrx-\ga\eta[\theta\phi+(1-\theta)\psi]}\|_r\}\times\\
&\qquad\qquad\times(\|\phi\|+\|\psi\|)(\|\phi-\psi\|).
\end{align*}
Similarly as above, taking $t,v>1$ such that $t^{-1}+v^{-1}=r^{-1}$, we estimate
\begin{align*}
\|\rho^2e^{\Wrx+\eta[\theta\phi+(1-\theta)\psi]}\|_r
\le&\|\rho^2e^{\Wrx}\|_t\|e^{\eta[\theta\phi+(1-\theta)\psi]}\|_v\\
\le&(C_1\rho^{2(1-t)/t}+C_2\rho^{(2-t)/t})e^{\frac{v}{16\pi}\||\phi|+|\psi|\|^2}\\
\le&(C_1\rho^{2(1-t)/t}+C_2\rho^{(2-t)/t})e^{\frac{v}{16\pi}(\|\phi\|+\|\psi\|)^2}.
\end{align*}
Choosing $2s=v=4/\eps$ so that $q^{-1}=r^{-1}+\eps/4$ and $t^{-1}=q^{-1}+\eps/4=r^{-1}+\eps/2$,
we conclude that
\begin{align*}
\|\Nrx(\phi)-\Nrx(\psi)\|_q\le(C_1\rho^{2(1-q)-\eps}+\rho^{(2-q)/q-\eps})
e^{\frac{1}{4\pi\eps}\||\phi|+|\psi|\|^2}(\|\phi\|+\|\psi\|)(\|\phi-\psi\|).
\end{align*}
For $\phi,\psi\in\mathcal B_{R_0|\log\rho|\rho^{(2-p)/p}}$ we thus obtain
\begin{align*}
\|\Nrx(\phi)-\Nrx(\psi)\|_q\le|\log\rho|\rho^{(2-p)/p}(C_1\rho^{2(1-q)/q-\eps})
e^{\frac{1}{4\pi\eps}\||\phi|+|\psi|\|^2}(\|\phi-\psi\|).
\end{align*}
By choosing $q>1$ and $\eps>0$ such that $2(1-q)/q-\eps+(2-p)/p>0$, we obtain that for $\rho$
sufficiently small $\Trx$ is indeed a contraction in $\mathcal B_{R_0|\log\rho|\rho^{(2-p)/p}}$, as asserted.
\end{proof}
\begin{proof}[Proof of Proposition~\ref{prop:phi}]
In view of Proposition~\ref{prop:contraction}, there exists $\rho_0>0$ such that the fixed point 
problem~\eqref{prob:phi} admits a solution $\phi_\rho\in\mathcal B_{R_0|\log\rho|\rho^{(2-p)/p}}$
for any $p\in(1,2)$ and for any $\rho\in(0,\rho_0)$.
Correspondingly, we obtain a solution for \eqref{eq:equ1ans}, which in turn
yields a solution for \eqref{equ1} satisfying Ansatz~\eqref{ans}.
\end{proof}
%%%%%%%%%%%%%%%%%%%%%%%%%%%%%%%%%%%%%%%%%%%%%%%%%%%%%%%%%%%%%%%%%%%%%%%%%%%%%%%%%%%%%%%%%%%%%%%%%%%%%%%
%%%%%%%%%%%%%%%%%%%%%%%%%%%%%%%%%%%%%%%%%%%%%%%%%%%%%%%%%%%%%%%%%%%%%%%%%%%%%%%%%%%%%%%%%%%%%%%%%%%%%%%
\section{The reduced problem}
\label{sec:reducedpb}
%%%%%%%%%%%%%%%%%%%%%%%%%%%%%%%%%%%%%%%%%%%%%%%%%%%%%%%%%%%%%%%%%%%%%%%%%%%%%%%%%%%%%%%%%%%%%%%%%%%%%%%
%%%%%%%%%%%%%%%%%%%%%%%%%%%%%%%%%%%%%%%%%%%%%%%%%%%%%%%%%%%%%%%%%%%%%%%%%%%%%%%%%%%%%%%%%%%%%%%%%%%%%%%
In this section we obtain $\xi_1,\xi_2\in\Omega$ such that equation \eqref{equ2} is fulfilled,
thus concluding the proof of Theorem~\ref{thm:main}.
Recall from \eqref{def:F} that $\mathcal H_\ga$ is defined by
\begin{equation*}
    \mathcal H_\ga(\xi_1,\xi_2):=H(\xi_1,\xi_1)+\frac{H(\xi_2,\xi_2)}{\ga^2}
-\frac{2G(\xi_1,\xi_2)}{\ga}.
    \end{equation*}
We consider the Euler-Lagrange functional for \eqref{p}, given by
$$
J_\rho(u):={1\over2}\int_\Omega|\nabla u|^2\,dx
-{\rho^2}\int_\Omega e^{u}\,dx
-{\rho^2\tau} \int_\Omega e^{-\ga u}\,dx
$$
for $u\in H^1_0(\Omega)$.
Then $u\in H^1_0(\Omega)$ is a solution for \eqref{p} if and only if it
is a critical point for $J_\rho$.
We define the ``reduced functional" $\widetilde J_\rho:\mathcal F_2\Omega\to\rr$
by setting
\begin{equation}
\label{def:Jred}
\widetilde J_\rho(\xi_1,\xi_2):= J_\rho\(\Wrx+\phi_\rho\),
\end{equation}
where, for every $(\xi_1,\xi_2)\in\mathcal F_2\Omega$, the
function $\phi_\rho$ is the solution to \eqref{equ1} obtained in Proposition~\ref{prop:phi}.
\par
The main result in this section is given by the following.
\begin{prop}
\label{prop:Jredu}
The function $u=\Wrx+\phi$ is a solution to problem~\eqref{p}
if and only if $(\xi_1,\xi_2)\in\mathcal F_2\Omega$ is a critical point for $\widetilde J_\rho$.
Moreover,  the following expansion holds true:
\begin{equation}
\label{eq:Jexp}
\begin{aligned}
\widetilde J_\rho(\xi_1,\xi_2)
=&-8\pi\left[\Big(1+\frac{1}{\ga^2}\Big)\log\rho^2
+\Big(\log\frac{1}{8}+1\Big)+\frac{1}{\ga^2}\Big(\log\frac{\tau\ga}{8}+1\Big)+1+\frac{1}{\ga}\right]\\
&\quad-\frac{(8\pi)^2}{2}\mathcal H_\ga(\xi_1,\xi_2)+o(1),
\end{aligned}
\end{equation}
$C^1-$uniformly in compact sets of $\mathcal F_2\Omega$.
\end{prop}
We first establish some lemmas.
\begin{lemma}
\label{lem:bubbleintegral}
For any $\de>0$ and $\xi\in\rr^2$, the Liouville bubble $w_\de$ satisfies
\begin{itemize}
  \item [(i)]
$\int_{B_\eps(\xi)}e^{w_\de}w_\de\,dx=8\pi(\log(8\de^{-2})-2)+O(\de^2\log\de)$
\item[(ii)]
$\int_{B_\eps(\xi)}e^{w_\de}Pw_\de\,dx
=8\pi(-2\log(\de^2)+8\pi H(\xi,\xi)-2)+O(\de)$.
\end{itemize}
for any fixed $\eps>0$.
\end{lemma}
\begin{proof}
Proof of (i).
We use the following identity, which is readily obtained by an integration by parts,
see also \cite{EGP}.
\begin{equation*}
\int_{\rr^2}\frac{1}{(1+|y|^2)^2}\log(1+|y|^2)\,dy=\pi
=\int_{\rr^2}\frac{1}{(1+|y|^2)^2}\,dy.
\end{equation*}
We compute
\begin{align*}
\int_{B_\eps(\xi)}e^{w_\de}w_\de\,dx
=&\int_{B_\eps(\xi)}\frac{8\de^2}{(\de^2+|x-\xi|^2)^2}\log\frac{8\de^2}{(\de^2+|x-\xi|^2)^2}\,dx\\
=&\int_{B_{\eps/\de}(0)}\frac{8}{\de^2(1+|y|^2)^2}\log\left(\frac{8}{\de^2(1+|y|^2)^2}\right)\,\de^2dy\\
=&8\log\frac{8}{\de^2}\int_{B_{\eps/\de}(0)}\frac{dy}{(1+|y|^2)^2}
-16\int_{B_{\eps/\de}(0)}\frac{1}{(1+|y|^2)^2}\log(1+|y|^2)\,dy\\
=&8\pi(\log(8\de^{-2})-2)+O(\de^2\log\de).
\end{align*}
\par
Proof of (ii). Recalling the expansion of $Pw_\de$, we have
\begin{align*}
\int_{B_\eps(\xi)}e^{w_\de}Pw_\de\,dx
=&\int_{B_\eps(\xi)}e^{w_\de}w_\de\,dx+(-\log(8\de^2)+8\pi H(\xi,\xi))\int_{B_\eps(\xi)}e^{w_\de}\,dx\\
&\quad+\int_{B_\eps(\xi)}e^{w_\de}O(|x-\xi|)\,dx+O(\de^2)\\
=&8\pi(\log(8\de^{-2})-2)+8\pi(-\log(8\de^2)+8\pi H(\xi,\xi))+O(\de)\\
=&8\pi(-2\log(\de^2)+8\pi H(\xi,\xi)-2)+O(\de).
\end{align*}
\end{proof}
Let
\begin{equation}
\label{def:ci}
c_1=-2\Big(\log\frac{1}{8}+1\Big),
\qquad
c_2=-2\Big(\log\frac{\tau\ga}{8}+1\Big).
\end{equation}
Using Lemma~\ref{lem:bubbleintegral} we readily derive the following. 
\begin{lemma}
\label{lem:bubbleintwithdelta}
Let $w_1=w_{\de_1,\xi_1}$, $w_2=w_{\de_2,\xi_2}$, with $\de_1,\de_2$ given by \eqref{eq:deltavalues}.
Then,
\begin{align*}
\int_{B_\eps(\xi_1)}e^{w_1}Pw_1=&8\pi(-2\log\rho^2-8\pi[H(\xi_1,\xi_1)-\frac{2}{\ga}\,G(\xi_1,\xi_2)]+c_1)+O(\rho)\\
\int_{B_\eps(\xi_2)}e^{w_2}Pw_2=&8\pi(-2\log\rho^2-8\pi[H(\xi_2,\xi_2)-2\ga\,G(\xi_1,\xi_2)]+c_2)+O(\rho)
\end{align*}
where the constants $c_i$, $i=1,2$, are defined in \eqref{def:ci}.
\end{lemma}
\begin{proof}
We compute, recalling Lemma~\ref{lem:bubbleintegral}, \eqref{pro-exp} 
and the definition of $\de_1$ in \eqref{eq:deltavalues}:
\begin{align*}
\int_{B_\eps(\xi_1)}e^{w_1}Pw_1
=&\int_{B_\eps(\xi_1)}e^{w_1}[w_1-\log(8\de_1^2)+8\pi H(\xi_1,\xi_1)+O(|x-\xi_1|+\rho^2)]\\
=&\int_{B_\eps(\xi_1)}e^{w_1}w_1+[-\log(8\de_1^2)+8\pi H(\xi_1,\xi_1)]\int_{B_\eps(\xi_1)}e^{w_1}\\
&\qquad+\int_{B_\eps(\xi_1)}e^{w_1}O(|x-\xi_1|+\rho^2)\\
=&8\pi[\log(\frac{8}{\de_1^2})-2]+8\pi[-\log(8\de_1^2)+8\pi H(\xi_1,\xi_1)]+O(\rho)\\
=&8\pi[-2\log\de_1^2-2+8\pi H(\xi_1,\xi_1)]+O(\rho)\\
=&8\pi[-2\log\Big(\frac{\rho^2}{8}e^{8\pi H(\xi,\xi)-\frac{8\pi}{\ga}G(\xi_1,\xi_2)}\Big)-2+8\pi H(\xi_1,\xi_1)]+O(\rho)\\
=&8\pi[-2\log\rho^2-8\pi H(\xi_1,\xi_1)+\frac{16\pi}{\ga}G(\xi_1,\xi_2)-2\log 8-2].
\end{align*}
This yields the expansion~(i). Expansion~(ii) is derived similarly.
\end{proof}
\begin{lemma}
\label{lem:JW}
The following expansion holds
\begin{align*}
\int_\Omega|\nabla\Wrx|^2\,dx=8\pi\left[-2(1+\frac{1}{\ga^{2}})\log\rho^2+c_1+\frac{c_2}{\ga^2}\right]
-(8\pi)^2\mathcal H_\ga(\xi_1,\xi_2)+O(\rho),
\end{align*}
where $c_i$, $i=1,2$, are defined in \eqref{def:ci}.
\end{lemma}
\begin{proof}
We have
\begin{equation*}
\int_\Omega|\nabla\Wrx|^2\,dx=\int_\Omega|\nabla Pw_1|^2\,dx
+\frac{1}{\ga^2}\int_\Omega|\nabla Pw_2|^2\,dx
-\frac{2}{\ga}\int_\Omega\nabla Pw_1\cdot\nabla Pw_2\,dx.
\end{equation*}
Integrating by parts, we obtain
\begin{align*}
\int_\Omega|\nabla Pw_i|^2\,dx=\int_\Omega(-\Delta Pw_i)Pw_i\,dx
=\int_\Omega e^{w_i}Pw_i\,dx,
\end{align*}
for $i=1,2$.
In view of Lemmas~\ref{lem:bubbleintegral}--\ref{lem:bubbleintwithdelta},
and observing that 
\begin{align*}
\int_\Omega\nabla Pw_1\cdot\nabla Pw_2\,dx=&\int_\Omega(-\Delta Pw_1)\,Pw_2\,dx=\int_\Omega e^{w_1}Pw_2\,dx\\
=&\int_\Omega e^{w_1}(8\pi G(\xi_1,\xi_2)+O(|x-\xi_1|)+O(\rho^2))\\
=&(8\pi)^2G(\xi_1,\xi_2)+O(\rho),
\end{align*}
we derive the asserted expansion.
\end{proof}
\begin{lemma}
\label{lem:masses}
The following asymptotics hold, as $\rho\to0$:
\begin{align*}
&\rho^2\int_\Omega e^{\Wrx}\,dx=8\pi+o(\rho);
&&\tau\rho^2\int_\Omega e^{-\ga\Wrx}=\frac{8\pi}{\ga}+o(\rho).
\end{align*}
\end{lemma}
\begin{proof}
We compute:
\begin{align*}
\rho^2\int_{B_\eps(\xi_1)}e^{\Wrx}
=&\rho^2\int_{B_\eps(\xi_1)}e^{w_1-\log(8\de_1^2)+8\pi H(\xi_1,\xi_1)-\frac{8\pi}{\ga}G(\xi_1,\xi_2)+O(\rho^2+|x-\xi_1|)}\\
=&\int_{B_\eps(\xi_1)}e^{w_1+O(\rho^2+|x-\xi|)}=8\pi+o(\rho).
\end{align*}
Similarly, we have
\begin{align*}
\tau\rho^2\int_{B_\eps(\xi_2)}e^{-\ga\Wrx}
=&\tau\rho^2\int_{B_\eps(\xi_2)}e^{w_2-\log(8\de_2^2)+8\pi H(\xi_2,\xi_2)-8\pi\ga\,G(\xi_1,\xi_2)+O(\rho^2+|x-\xi_2|)}\\
=&\frac{1}{\ga}\int_{B_\eps(\xi_2)}e^{w_2+O(\rho^2+|x-\xi_2|)}=\frac{8\pi}{\ga}+o(\rho).
\end{align*}
\end{proof}
\begin{proof}[Proof of Proposition~\ref{prop:Jredu}]
Similarly as in \cite{BP, EGP}, we readily check that 
$$
\widetilde J_\rho(\xi_1,\xi_2)=J_\rho(\Wrx)+O(\|\phi_\rho\|)
$$
in $C^0$, on compact subsets of $\mathcal F_2\Omega$.
In turn, Lemma~\ref{lem:JW} yields the $C^0-$ uniform 
convergence of $\widetilde J_\rho$ to the functional on the r.h.s.\ of \eqref{eq:Jexp} 
on compact subsets of $\mathcal F_2\Omega$.
The $C^1-$ uniform convergence on compact subsets of $\mathcal F_2\Omega$
may be then derived by a step-by-step adaptation of the arguments in \cite{EGP},
which rely on an implicit function argument and on the invertibility 
the operator $\mathcal L_\rho^\xi$ as stated in Proposition~\ref{prop:Linvert}.
\par
We are left to show that
critical points of $\widetilde J_\rho$ correspond to
critical points of $J_\rho$.
To this end, we observe that since $u_\rho$ satisfies  \eqref{equ1}, 
there exist constants $c_{ih}$, $i,h=1,2$
such that
\begin{equation}
u_\rho-i^*[f_\rho(u_\rho)]=\sum_{i,h=1}^2c_{ih}P\psi_i^h,
\end{equation}
where the functions $\psi_i^h$ are defined in \eqref{def:psi}.
Therefore, we may write
\begin{equation}
\label{eq:tildeJexp}
\begin{aligned}
\partial_{\xi_{11}}\widetilde J_\rho(\xi_1,\xi_2)=\langle J_\rho'(u_\rho),\partial_{\xi_{11}}u_\rho\rangle
=&(u_\rho-i^*[f_\rho(u_\rho)],\partial_{\xi_{11}}(\Wrx+\phi_\rho))_{H_0^1(\Omega)}\\
=&(\sum_{i,h=1}^2c_{ih}P\psi_i^h,\partial_{\xi_{11}}\Wrx)_{H_0^1(\Omega)}.
\end{aligned}
\end{equation}
On the other hand, by definition of $\Wrx$ we have
\begin{equation*}
\partial_{\xi_{11}}\Wrx=\partial_{\xi_{11}}Pw_1-\frac{1}{\ga}Pw_2
=P\psi_1^1+P\psi_1^0\partial_{\xi_{11}}\de_1(\xi_1,\xi_2)-\frac{1}{\ga}P\psi_2^0\partial_{\xi_{11}}\de_2(\xi_1,\xi_2).
\end{equation*}
In view of \eqref{eq:psiorthog} and observing that $\partial_{\xi_{11}}\de_i(\xi_1,\xi_2)=O(\rho)$, $i=1,2$,
we conclude that
\begin{equation*}
(\sum_{i,h=1}^2c_{ih}P\psi_i^h,\partial_{\xi_{11}}\Wrx)_{H_0^1(\Omega)}
=c_{11}\frac{\de D}{\rho^3}(1+O(\rho)).
\end{equation*}
Now it follows from \eqref{eq:tildeJexp} and the above that if $\partial_{\xi_{11}}\widetilde J_\rho(\xi_1,\xi_2)=0$
then necessarily $c_{11}=0$. Similarly, we check that $c_{12}=c_{21}=c_{22}=0$.
\end{proof}
\begin{proof}[Proof of Theorem~\ref{thm:main}]
We use standard Ljusternik-Schnirelmann theory to obtain
$\mathrm{cat}\mathcal F_2\Omega$ critical points for $\mathcal H_\ga(\xi_1,\xi_2)$.
More precisely, we note that $\mathcal H_\ga(\xi_1,\xi_2)\to-\infty$
as $(\xi_1,\xi_2)\to\partial\mathcal F_2\Omega$.
Consequently, $\mathcal H_\ga$ is bounded from above on $\mathcal F_2\Omega$
and we may apply Theorem~2.3 in \cite{Ambrosetti}
to derive the asserted existence of critical points $(\xi_1^i,\xi_2^i)\in\mathcal F_2\Omega$.
See also Theorem~2.1 in \cite{BaPi}.
Since $\widetilde J_\rho\to\mathcal H_\ga$ in $C^1(\mathcal F_2\Omega)$,
we conclude for sufficiently small values of $\rho$ the functional
$\widetilde J_\rho$ admits at least $\mathrm{cat}\mathcal F_2\Omega$ critical points
$(\xi_{\rho,1}^i,\xi_{\rho,2}^i)\to(\xi_1^i,\xi_2^i)$, $i=1,\ldots,\mathrm{cat}\mathcal F_2\Omega$.
For each fixed $i=1,\ldots,\mathrm{cat}\mathcal F_2\Omega$, we then apply Proposition~\ref{prop:phi}
with $(\xi_1,\xi_2)=(\xi_{\rho,1}^i,\xi_{\rho,2}^i)$
to obtain the desired solutions $u_{\rho}^i$.
\par
Proof of (i). By construction, $u_{\rho}^i$, $i=1,\ldots,\mathrm{cat}\mathcal F_2\Omega$
satisfies Ansatz~\eqref{ans}.
\par
Proof of (ii).
We adapt an argument from \cite{BP,BaMiPi} to our situation.
Since $\|\phi_\rho\|_{L^\infty}\to0$ as $\rho^2\to0$, there exist disjoint balls $B_r(\xi_{i,\rho})\subset \Omega\setminus\{x\in\Omega:\ u_\rho(x)=0\}$,
$i=1,2$, $\mathrm{dist}(B_r(\xi_{1,\rho}),B_r(\xi_{2,\rho}))\ge\delta>0$ such that $u_\rho\ge\delta$ in $B_r(\xi_{1,\rho})$
and $u_\rho\le-\delta$ in $B_r(\xi_{2,\rho})$.
Therefore, the set $\Omega\setminus\{x\in\Omega:\ u_\rho(x)=0\}$ has at least two connected components.
Arguing by contradiction, we assume that there exists
another connected component $\Omega_\rho\subset\Omega\setminus\{x\in\Omega:\ u_\rho(x)=0\}$
with the property $\omega_\rho\supset B_r(\xi)$ for some $B_r(\xi)\in\Omega\setminus(B_r(\xi_{1,\rho})\cup B_r(\xi_{2,\rho}))$.
Then $u_\rho$ satisfies
$$
\left\{
\begin{aligned}
-\Delta u_\rho=&a_\rho u_\rho+\rho^2(1-\tau)\qquad\mathrm{in\ }\omega_\rho\\
u_\rho\in&H_0^1(\omega_\rho) 
\end{aligned}
\right.
$$
with $a_\rho$ defined by
$$
a_\rho=\frac{f_\rho(u_\rho)-\rho^2(1-\tau)}{u_\rho}
=\rho^2\,\frac{(e^{u_\rho}-1)-\tau(e^{-\ga u_\rho}-1)}{u_\rho}.
$$ 
Multiplying by $u_\rho$ and integrating, we obtain
\begin{align*}
\int_{\omega_\rho}|\nabla u_\rho|^2\,dx\le&\|a_\rho\|_{L^\infty(\Omega_\rho)}\|u_\rho\|_{L^2(\omega_\rho)}^2
+\rho^2|\tau_1-\tau_2|\|u_\rho\|_{L^1(\omega_\rho)}\\
\le&\frac{\|a_\rho\|_{L^\infty(\omega_\rho)}}{\lambda_1(\omega_\rho)}\int_{\omega_\rho}|\nabla u_\rho|^2\,dx
+\rho^2\,\frac{|1-\tau|\,|\omega_\rho|^{1/2}}{\la_1^{1/2}(\omega_\rho)}\left(\int_{\omega_\rho}|\nabla u_\rho|^2\,dx\right)^{1/2}
\end{align*}
where for any $\omega\subset\Omega$ we denote by $\la_1(\omega)$ the first eigenvalue of the operator $-\Delta$
defined on $\omega$,
subject to Dirichlet boundary conditions.
Recalling that $\la_1(\omega_\rho)\ge\la_1(\Omega)>0$, we derive that
\begin{align}
\label{eq:omegarho}
\left(1-\frac{\|a_\rho\|_{L^\infty(\omega_\rho)}}{\la_1(\Omega)}\right)\left(\int_{\omega_\rho}|\nabla u_\rho|^2\,dx\right)^{1/2}
\le\rho^2\,\frac{|1-\tau|\,|\omega_\rho|^{1/2}}{\la_1^{1/2}(\Omega)}.
\end{align}
Since $u_\rho\to u_0$ in $C^2(\overline{\omega_\rho})$,
with $u_0=8\pi G(\cdot,\xi_1^*)-8\pi\ga^{-1}G(\cdot,\xi_2^*)$
for some $\xi_1^*,\xi_2^*\in\Omega$, $\xi_1^*\neq\xi_2^*$,
we have
$$
\int_{\omega_\rho}|\nabla u_\rho|^2\,dx\ge\int_{B_r(\xi)}|\nabla u_0|^2\,dx>0.
$$
On the other hand, we have $\|a_\rho\|_{L^\infty(\Omega_\rho)}=O(\rho^2)$.
We thus obtain from \eqref{eq:omegarho} that $(1+O(\rho^2))\le C\rho^2$,
a contradiction.
\par
Proof of (iii). The proof is a straightforward consequence of the symmetry of the problem. 
See also Theorem~2.1 in \cite{BaPi},
Part (b).
\end{proof}
%%%%%%%%%%%%%%%%%%%%%%%%%%%%%%%%%%%%%%%%%%%%%%%%%%%%%%%%%%%%%%%%%%%%%%%%%%%%%%%%%%%%%%%%%%%%
%%%%%%%%%%%%%%%%%%%%%%%%%%%%%%%%%%%%%%%%%%%%%%%%%%%%%%%%%%%%%%%%%%%%%%%%%%%%%%%%%%%%%%%%%%%%
\section{Proof of Theorem~\ref{thm:bubbleescape}}
\label{sec:bubbleescape}
%%%%%%%%%%%%%%%%%%%%%%%%%%%%%%%%%%%%%%%%%%%%%%%%%%%%%%%%%%%%%%%%%%%%%%%%%%%%%%%%%%%%%%%%%%%%
%%%%%%%%%%%%%%%%%%%%%%%%%%%%%%%%%%%%%%%%%%%%%%%%%%%%%%%%%%%%%%%%%%%%%%%%%%%%%%%%%%%%%%%%%%%%
In this section we prove Theorem~\ref{thm:bubbleescape} by carefully analyzing the asymptotic behavior
of the critical points of 
the Hamiltonian $\mathcal H_\ga$ defined in \eqref{def:F}.
For the sake of simplicity, we slightly change notation throughout this section.
We recall that
$$
{\mathcal H}_\ga(x,y):=h(x)+\frac{h(y)}{\ga^2}-\frac{2}{\ga} G(x,y),  
$$
for all $(x,y)\in\Omega\times\Omega,\ x\not= y$,
where
$$
G(x,y):={1\over 2\pi}\ln{1\over|x-y|} +H(x,y)
$$
is the Green's function and we denote by
$$
h(x):=H(x,x)
$$
the Robin's function.
\par
Our aim in this section is to establish the following result, which is the main ingredient needed
in the proof of Theorem~\ref{thm:bubbleescape}.
\begin{prop}
\label{prop:Hamiltonian}
Let $\Omega\subset\mathbb R^2$ be a convex bounded domain. Let $\ga_n\to+\infty.$
Let $(x_n,y_n)$ be a critical point of ${\mathcal H}_{\ga_n}$ such that $(x_n,y_n)\to(x_0,y_0)\in\bar\Omega\times\bar\Omega$.
Then, we have:
\begin{enumerate}
\item [(i)]
$x_0\in\Omega$; moreover, $x_0$ is the unique maximum point of the Robin's function;
\item[(ii)]
$y_0\in\partial\Omega$; moreover, $y_0$ is a minimum point of the function $\partial_\nu G(x_0,y),$ $y\in\partial\Omega$.
Here $\nu$ denotes the outward normal at $y\in\partial\Omega$.
\end{enumerate}
\end{prop}
We collect in the following lemmas some known results which are needed in the proof of Proposition~\ref{prop:Hamiltonian}.
We first introduce some notation.
For a fixed small constant $\eps_0>0$ we define the tubular neighborhood
$$
\Omega_0:=\{x\in\Omega:\ \mathrm{dist}(x,\partial\Omega)<\eps_0\}.
$$
We assume that $\eps_0$ is sufficiently small so that the reflection map at $\partial\Omega$,
denoted by
$x\in\Omega_0\mapsto\bar x\in\rr^2\setminus\overline\Omega$, is well-defined.
Correspondingly, we define the orthogonal projection $p:\Omega_0\to\partial\Omega$ by setting
$p(x)=(x+\bar x)/2$. The outward normal at $p(x)$ is then given by
$(\bar x-x)/|\bar x-x|$.
For $x\in\Omega$ we denote $d_x=\mathrm{dist}(x,\partial\Omega)$.
\begin{lemma}
\label{lem:Gh}
The following properties hold for the Green's function $G(x,y)$ and the Robin's function $h(x)$.
\begin{enumerate}
\item[(i)]
[\cite{CaffarelliFriedman1985}, Theorem~3.1.] 
Let $\Omega\subset\rr^2$ be a convex domain, not necessarily bounded,
which is not an infinite strip, and let $h=h_\Omega$ denote the associated Robin's function.
Then, $-h$ is strictly convex,
that is, the Hessian $(-h_{ij})$ is strictly positive definite. 
\item[(ii)]
[\cite{BDP}, Lemma~A.2] 
Let $\Omega\subset\rr^N$, $N\ge2$, be a convex bounded domain.
Then for any $x,y\in\Omega$, $x\neq y$, we have
\begin{equation*}
(x-y)\cdot\nabla_xG(x,y)<0.
\end{equation*}
\item[(iii)]
[\cite{BF}, p.~204] 
Suppose $\partial\Omega$ is sufficiently smooth so that $e^h\in C^2(\Omega)$.
Then, writing $y=p(y)-d_y\nu(y)$ for $y\in\Omega_0$, the following expansion holds:
\begin{equation*}
h(y)=\frac{1}{2\pi}\left(\log(2d_y)-\frac{\kappa(p(y))}{2}d_y+o(d_y)\right),
\end{equation*}
where $\kappa$ denotes the mean curvature of the boundary with respect to the exterior normal.
\end{enumerate}
\end{lemma}
\begin{rmk}
Although Lemma~A.2 in \cite{BDP} is stated for $N\ge3$, it is clear that it holds for $N=2$
as well, in view of \cite{GT}.
\end{rmk}
Exploiting the explicit expression
of the Green's function for the half-plane, the following accurate expansions may be derived.
\begin{lemma}[\cite{BaPi}, Lemma~3.2]
\label{lem:Ghexpansions}
Let $(x_n,y_n)\in\Omega\times\Omega$. Then,
\begin{enumerate}
\item[(i)]
$h(x_n)=\frac{1}{2\pi}\log(2d_{x_n})+O(1),\quad d_{x_n}|\nabla h(x_n)|=O(1)$,
if $x_n\in\Omega_0$;
\item[(ii)]
$\nabla h(x_n)={1\over2\pi d_{x_n}}\nu(x_n)+o(1)$, if $d_{x_n}\to0$;
\item[(iii)]
$\nabla _xG(x_n,y_n)=-{1\over2\pi}{x_n-y_n\over |x_n-y_n|^2 }+O\({1\over d{x_n}}\)$, if $x_n\in\Omega_0$;
\item[(iv)]
$\langle\nabla_{x_n}G(x_n,y_n),\nu(x_n)\rangle
+\langle\nabla_{y_n}G(y_n,x_n),\nu(y_n)\rangle
=\frac{1}{2\pi}(d_{x_n}+d_{y_n})\Big(\frac{1}{|\bar x_n-y_n|^2}+\frac{1}{|\bar y_n-x_n|^2}\Big)+O(1)$,
if $x_n,y_n\in\Omega_0$.
\item[(v)]
$|\bar x_n-y_n|^2=|x_n-y_n|^2+4d_{x_n}d_{y_n}+\circ(|x_n-y_n|^2)$, if $x_n,y_n\to p^*\in\partial\Omega$. 
\end{enumerate}
\end{lemma}
\begin{proof}[Proof of Proposition~\ref{prop:Hamiltonian}]
By assumption, $(x_n,y_n)$ is a critical point of $\mathcal H_{\ga_n}$, that is:
\begin{eqnarray}
\label{1}
&\ga_n\nabla h(x_n)=\nabla_xG(x_n,y_n)\\
\label{2}
&\frac 1{\ga_n}\nabla h(y_n)=\nabla_yG(x_n,y_n).
\end{eqnarray}
We first establish the following.
\par
\textit{Claim 1: $x_0\not=y_0$.}
\par
Indeed, suppose the contrary.
\par
We first consider the case $x_0=y_0\in  \Omega$. Then,
$\nabla h(y_n)=O(1)$. Consequently, \eqref{2} implies that
$\nabla_y G(x_n,y_n)=o(1),$ a contradiction.
\par
Hence, we consider the case $x_0=y_0\in \partial\Omega$. 
We claim that
\begin{equation}
\label{7}
{|x_n-y_n|\over d_{x_n}}=o(1).
\end{equation}
Indeed, if not we may assume that ${d_{x_n}\over|x_n-y_n| }=O(1)$.
Multiplying \eqref{1} by $\nu(x_n)$, using Lemma~\ref{lem:Ghexpansions}--(i)--(iii)  we deduce
\begin{equation*}
\begin{aligned}
\ga_n(\frac{1}{2\pi d_{x_n}}+o(1))=&\ga_n\langle\nabla h(x_n),\nu(x_n)\rangle
=\langle\nabla_xG(x_n,y_n),\nu(x_n)\rangle\\
=&-\frac{1}{4\pi}\frac{\langle x_n-y_n,\nu(x_n)\rangle}{|x_n-y_n|^2}+O(\frac{1}{d_{x_n}})
\end{aligned}
\end{equation*}
and therefore
\begin{equation*}
\label{6}
1=-{d_{x_n}\over\ga_n}{\langle x_n-y_n,\nu(x_n)\rangle\over 2|x_n-y_n|^2} +o(1).
\end{equation*}
In turn, we find  
$$
1=O\({d_{x_n}\over\ga_n|x_n-y_n|}\)=o(1),
$$
and a contradiction arises. Therefore, \eqref{7} is established.
\par
Similarly, we claim that 
\begin{equation}
\label{8}
  {d_{y_n}\over d_{x_n}}=o(1).
\end{equation}
Indeed, if not we may assume that $ {d_{x_n}\over d_{y_n}}=O(1)$.  Multiplying \eqref{1} 
by  $\nu(x_n)$ and \eqref{2} by  $\nu(y_n)$,  and adding the two identities 
we obtain
\begin{align*}
\ga_n\langle\nabla h(x_n),\nu(x_n)\rangle+\frac{1}{\ga_n}\langle\nabla h(y_n),\nu(y_n)\rangle
=\langle\nabla_xG(x_n,y_n),\nu(x_n)\rangle+\langle\nabla_yG(x_n,y_n),\nu(y_n)\rangle.
\end{align*}
Hence, using Lemma~\ref{lem:Ghexpansions}--(ii)--(iv)
we derive that
\begin{equation*}
\ga_n\(\frac{1}{2\pi d_{x_n}}+o(1)\)+\frac{1}{\ga_n}\(\frac{1}{2\pi d_{y_n}}+o(1)\)
=\frac{1}{2\pi}(d_{x_n}+d_{y_n})\(\frac{1}{|\bar x_n-y_n|^2}+\frac{1}{|\bar y_n-x_n|^2}\).
\end{equation*}
In turn, using Lemma~\ref{lem:Ghexpansions}--(v) 
we deduce
$$
{\ga_n\over d_{x_n}}+{1\over  \ga_nd_{y_n}}
=O\({d_{x_n}+d_{y_n}\over d_{x_n}d_{y_n}}\). 
$$
The above yields
$$
{1}+{1\over\ga^2_n}{d_{x_n}\over d_{y_n}}
={1\over\ga_n}O\({d_{x_n}\over d_{y_n}}+1\)  
$$
and a contradiction arises. Therefore, \eqref{8} is established.
\par
Finally, \eqref{7}--\eqref{8} and the triangle inequality
$$
d_{x_n}\le |x_n-y_n|+d_{y_n}
$$
yield a contradiction.
Hence, the proof of Claim~1 is complete.
\par
\textit{Claim~2: $x_0\in\Omega$ and $y_0\in\partial\Omega$.}
\par
Since $x_0\neq y_0$ in view of Claim~1, we have
\begin{equation}
\label{3}
\nabla G(x_n,y_n)=O(1).
\end{equation}
If $x_0\in \partial\Omega$ then 
 $|\nabla h(x_n)|\to+\infty$ and by \eqref{1} and \eqref{3} we get a contradiction.
If $x_0\in  \Omega$ and  $y_0\in \Omega$ then 
$\nabla H(y_n)=O(1)$ and   by \eqref{2} we deduce that 
$\nabla _y G(x_0,y_0)=0.$ This is impossible if $\Omega$ is convex, in view of 
Lemma~\ref{lem:Gh}--(ii).
Hence, Claim~2 is established.
\par
Proof of (i).
We are left to show that $x_0$ is the maximum point of the Robin's function.
Since $x_0\in\Omega$ and  $y_0\in \partial\Omega$, then by \eqref{3} and \eqref{1} we derive
$\nabla h(x_0)=0.$
Since the domain is bounded and convex, in view of Lemma~\ref{lem:Gh}--(i)
Robin's function has a unique critical point, given by the maximum point. 
Now Proposition~\ref{prop:Hamiltonian}--(i)
is completely established.
\par
Proof of (ii).
By the mean value theorem we may write for any $x\in\Omega$
\begin{equation}
\label{mv}
G(x,y)=G(x,p(y)-d_y\nu(y))=-\partial_\nu G(x,p(y))d_y+o(d_y).
\end{equation} 
Let $(x_n,y_n)$ be the maximum point of the function $\mathcal H_n.$ 
For any point $p\in\partial\Omega,$ we consider $y=p-d_{y_n}\nu(p)\in\Omega.$ Then, we have
$\mathcal H_n(x_n,y_n)\ge \mathcal H_n(x_n,y)$.  
That is,
$$
h(y_n)-2\ga_n G(x_n,y_n)\ge h(y)-2\ga _n G(x_n,y).
$$
In view of Lemma~\ref{lem:Gh}--(iii) and \eqref{mv} we derive
$$
-{1\over2\pi}{\kappa(p(y_n))\over2}d_{y_n}-2\ga_n \partial_\nu G(x_n,p(y_n))d_{y_n}+o(d_{y_n})
\ge-{1\over2\pi}{\kappa(p)\over2}d_{y_n}-2\ga_n \partial_\nu G(x_n,p)d_{y_n}+o(d_{y_n}).
$$
Recalling that $\ga_n\to+\infty$, we derive from the above that
$$
\partial_\nu G(x_n,p(y_n))\le\partial_\nu G(x_n,p)+o(1).
$$
Finally, taking limits, we obtain
$$
\partial_\nu G(x_0,y_0) \le\partial_\nu G(x_0,p)
$$ 
for any $p\in\partial\Omega$, and (ii) is completely established.
\end{proof}
Finally, we provide the proof of
Theorem~\ref{thm:bubbleescape}.
\begin{proof}[Proof of Theorem~\ref{thm:bubbleescape}]
Proof of (i). Let $\ga_n\to+\infty$.
The asserted asymptotic behavior follows readily from Proposition~\ref{prop:Hamiltonian}
with $(x_n,y_n)=(\xi_1^{\ga_n},\xi_2^{\ga_n})$.
Proof of (ii). In this case, we take $(x_n,y_n)=(\xi_2^{\ga_n},\xi_1^{\ga_n})$.
\end{proof}
%%%%%%%%%%%%%%%%%%%%%%%%%%%%%%%%%%%%%%%%%%%%%%%%%%%%%%%%%%%%%%%%%%%%%%%%%%%%%%%%%%%%
%%%%%%%%%%%%%%%%%%%%%%%%%%%%%%%%%%%%%%%%%%%%%%%%%%%%%%%%%%%%%%%%%%%%%%%%%%%%%%%%%%%
\section{Appendix}
%%%%%%%%%%%%%%%%%%%%%%%%%%%%%%%%%%%%%%%%%%%%%%%%%%%%%%%%%%%%%%%%%%%%%%%%%%%%%%%%%%%
%%%%%%%%%%%%%%%%%%%%%%%%%%%%%%%%%%%%%%%%%%%%%%%%%%%%%%%%%%%%%%%%%%%%%%%%%%%%%%%%%%%
We provide in this section a blow-up analysis for solution sequences to \eqref{p},
from which it is clear that the blow-up masses and the locations of the blow-up points, as taken
in Theorem~\ref{thm:main}, are the only possible choice.
\begin{prop}
\label{prop:blowupanalysis}
Assume that $u_{\rho_n}$ is a solution sequence for \eqref{p}
satisfying $u_{\rho_n}\to u_0$ in $C_{\mathrm{loc}}^2(\Omega\setminus\{\xi_1,\xi_2\})\cap W_0^{1,p}(\Omega)$, $1\le p<2$,
with
$$
u_0(x)=n_1G(x,\xi_1)-n_2G(x,\xi_2)
$$
for some $\xi_1,\xi_2\in\Omega$ and for some $n_1,n_2>0$.
Then,
\begin{align}
\label{eq:mass}
&\ n_1=8\pi,&&\ n_2=\frac{8\pi}{\ga}\\
\label{eq:location}
&\left.\nabla_\xi\left[H(\xi,\xi_1)-\frac{G(\xi,\xi_2)}{\ga}\right]\right\vert_{\xi=\xi_1}=0,
&&\left.\nabla_\xi\left[\frac{H(\xi,\xi_2)}{\ga}-G(\xi,\xi_1)\right]\right\vert_{\xi=\xi_2}=0
\end{align}
\end{prop}
\begin{proof}
We adapt a technique from \cite{Ye}. For the sake of simplicity, throughout this proof, we denote $u=u_{\rho_n}$.
We recall that 
\begin{align*}
f_\rho(t)=&\rho^2(e^{t}-\tau e^{-\ga t})\\
F_\rho(t)=&\rho^2\left(e^{t}+\frac{\tau}{\ga}e^{-\ga t}\right)
\end{align*}
so that $-\Delta u=f_\rho(u)$ and $F_\rho'=f_\rho$. By assumption, we have 
$f_\rho(u)\stackrel{\ast}{\rightharpoonup}n_1\delta_{\xi_1}-n_2\delta_{\xi_2}$
weakly in the sense of measures, and therefore $\rho^2e^{u}\stackrel{\ast}{\rightharpoonup}n_1\delta_{\xi_1}$
and $\rho^2\tau e^{-\ga u}\stackrel{\ast}{\rightharpoonup}n_2\delta_{\xi_2}$.
It follows that
\begin{equation}
\label{eq:Flimit}
F_\rho(u)\stackrel{\ast}{\rightharpoonup}n_1\delta_{\xi_1}+\frac{n_2}{\ga}\delta_{\xi_2},
\end{equation}
weakly in the sense of measures.
Using the standard complex notation $z=x+iy$, $\partial_z=(\partial_x-i\partial_y)/2$, $\partial_{\bar z}=(\partial_x+i\partial_y)/2$
so that $\partial_{z\bar z}=\Delta/4$,
we define the quantities
\begin{align*}
H=&\frac{1}{2}u_z^2\\
K=&N_{zz}\ast F_\rho(u)=N_z\ast[F_\rho(u)]_z,
\end{align*}
where $N(z,\bar z)=(4\pi)^{-1}\log(z\bar z)$ is the Newtonian potential. We note that $\Delta N=\delta_0$,
$N_z=(4\pi z)^{-1}$, $N_{zz}=-(4\pi z^2)^{-1}$.
Let $S=H+K$. It is readily checked that $S_{\bar z}=0$, that is, $S$ is holomorphic in $\Omega$.
Indeed, we have $H_{\bar z}=-4u_zf_\rho(u)$ and 
$K_{\bar z}=N_{z\bar z}\ast[F_\rho(u)]_z=4f_\rho(u)u_z$.
It follows that $S$ converges to some holomorphic function $S_0$.
In order to determine $S_0$, we separately take limits for $H$ and $K$.
By assumption, we have
\begin{equation}
H\to H_0=\frac{1}{2}u_{0,z}^2=\frac{1}{2}[n_1G_z(z,\xi_1)-n_2G_z(z,\xi_2)]^2
\end{equation}
in $C_\mathrm{loc}^2(\Omega\setminus\{\xi_1,\xi_2\})$.
Recalling that 
$G(z,\xi)=(4\pi)^{-1}\log[(z-\xi)(\bar z-\bar\xi)]+H(z,\xi)$,
we derive
$G_z(z,\xi)=(4\pi(z-\xi))^{-1}+H_z(z,\xi)$.
Hence, we may write
$$
u_{0,z}=\frac{n_1}{4\pi(z-\xi_1)}-\frac{n_2}{4\pi(z-\xi_2)}+\omega_z,
$$
where the function
$$
\omega_z(z)=n_1H(z,\xi_1)-n_2H(z,\xi_2).
$$
is smooth in $\Omega$. Thus, we derive
\begin{equation}
\label{eq:H0}
\begin{aligned}
H_0=\frac{n_1^2}{32\pi^2(z-\xi_1)^2}&+\frac{n_2^2}{32\pi^2(z-\xi_2)^2}
-\frac{n_1n_2}{16\pi^2(z-\xi_1)(z-\xi_2)}\\
&+\frac{n_1}{4\pi(z-\xi_1)}\omega_z-\frac{n_2}{4\pi(z-\xi_2)}\omega_z+\frac{1}{2}\omega_z^2.
\end{aligned}
\end{equation}
On the other hand, we have $K\to K_0$, with
\begin{equation}
\label{eq:K0}
\begin{aligned}
K_0=&N_{zz}\ast\left[n_1\delta_{\xi_1}+\frac{n_2}{\ga}\delta_{\xi_2}\right]
=-\frac{1}{4\pi z^2}\ast\left[n_1\delta_{\xi_1}+\frac{n_2}{\ga}\delta_{\xi_2}\right]\\
=&-\frac{n_1}{4\pi(z-\xi_1)^2}-\frac{n_2}{4\pi\ga(z-\xi_2)^2}.
\end{aligned}
\end{equation}
Since $S_0=H_0+K_0$ is holomorphic, the singularities of $H_0$ necessarily cancel the singularities of $K_0$.
Cancellation of the second-order singularities readily yields $n_1=8\pi$. The second identity in \eqref{eq:mass},
namely $n_2=8\pi/\ga$ is derived similarly.
\par
Now, we consider the first-order singularities.
Near $\xi_1$, we obtain that
$$
-\frac{n_1n_2}{16\pi^2(\xi_1-\xi_2)}+\frac{n_1\omega_z(\xi_1)}{4\pi}=0.
$$
That is, using \eqref{eq:mass},
\begin{equation}
\label{eq:omegaz}
\omega_z(\xi_1)=\left.\left[8\pi H_z(z,\xi_1)-\frac{8\pi}{\ga}H_z(z,\xi_2)\right]\right\vert_{z=\xi_1}
=\frac{2}{\ga(\xi_1-\xi_2)}.
\end{equation}
On the other hand, we may write
\begin{align*}
\frac{1}{4\pi(\xi_1-\xi_2)}=\frac{1}{4\pi}\partial_z\log[(z-\xi_2)(\bar z-\bar\xi_2)]\vert_{z=\xi_1}
=N_z(z,\xi_2)=G_z(z,\xi_2)-H_z(z,\xi_2).
\end{align*}
Therefore, in view of \eqref{eq:omegaz} we obtain
\begin{align*}
\left.\left[H_z(z,\xi_1)-\frac{H_z(z,\xi_2)}{\ga}\right]\right\vert_{z=\xi_1}
=\frac{1}{4\pi\ga(\xi_1-\xi_2)}
=\frac{1}{\ga}G_z(z,\xi_2)\vert_{z=\xi_1}-\frac{H_z(z,\xi_2)}{\ga}.
\end{align*}
Hence, we conclude that
$$
\partial_z\left.\left[H(z,\xi_1)-\frac{G(z,\xi_2)}{\ga}\right]\right\vert_{z=\xi_1}=0.
$$
Since $H,G$ are real, the first equation in \eqref{eq:location} follows.
The second equation in \eqref{eq:location} is derived similarly.
\end{proof}
We note that \eqref{eq:location} implies that at the blow-up points $\xi_1,\xi_2$ 
the estimate in
Lemma~\ref{lem:deltachoice} may be improved as follows.
\begin{lemma}
\label{lem:phiimproved}
Let $\xi_1,\xi_2$ satisfy \eqref{eq:location} and let $\Wrx$ be defined by \eqref{ans}.
Then,
\begin{equation}
\label{eq:basicestimproved}
\left\|\rho^2e^{\al_1W_\rho^\xi}-e^{w_1}\right\|_{L^p(\Omega)}^p
+\left\|\rho^2\tau\ga e^{-\ga W_\rho^\xi}-e^{w_2}\right\|_{L^p(\Omega)}^p
\le C\rho^{2}.
\end{equation}
\end{lemma}
\begin{proof}
In view of \eqref{eq:location}, the Taylor expansion employed
in the proof of Lemma~\ref{lem:deltachoice} may be improved:
\begin{equation*}
H(x,\xi_1)-\frac{1}{\ga}G(x,\xi_2)=H(\xi_1,\xi_1)-\frac{1}{\ga}G(\xi_1,\xi_2)+O(|x-\xi_1|^2).
\end{equation*}
Consequently, we estimate
\begin{align*}
\int_{B_\eps(\xi_1)}&|\rho^2e^{\al_1\Wrx}-e^{w_1}|^p\,dx\\
=&\int_{B_\eps(\xi_1)}\left|\rho^2e^{w_1-\log(8\delta_1^2)
+8\pi[H(\xi_1,\xi_1)-\frac{1}{\ga}G(\xi_1,\xi_2)]+O(\delta_1^2+|x-\xi_1|^2)}-e^{w_1}\right|^p\,dx\\
\le&C\int_{B_\eps(\xi_1)}e^{pw_1}(\delta_1^2+|x-\xi_1|^2)^p\,dx
\le C\de_1^{2p}\int_{B_\eps(\xi_1)}\frac{dx}{(\de_1^2+|x-\xi_2|^2)^{p}}\,dx\\
\le&C\de_1^{2}\int_{B_{\eps/\de_1}(0)}\frac{dy}{(1+|y|^2)^{p}}\,dy
\le C\rho^{2}.
\end{align*}
At this point, arguing as in Lemma~\ref{lem:deltachoice}, we conclude the proof.
\end{proof}
%%%%%%%%%%%%%%%%%%%%%%%%%%%%%%%%%%%%%%%%%%%%%%%%%%%%%%%%%%%%%%%%%%%%%%%%%%%%%%%%%%%
%%%%%%%%%%%%%%%%%%%%%%%%%%%%%%%%%%%%%%%%%%%%%%%%%%%%%%%%%%%%%%%%%%%%%%%%%%%%%%%%%%%
%%%%%%%%%%%%%%%%%%%%%%%%%%%%%%%%%%%%%%%%%%%%%%%%%%%%%%%%%%%%%%%%%%%%%%%%%%%%%%%%%%%
%%%%%%%%%%%%%%%%%%%%%%%%%%%%%%%%%%%%%%%%%%%%%%%%%%%%%%%%%%%%%%%%%%%%%%%%%%%%%%%%


\begin{thebibliography}{99}
\bibitem{Ambrosetti}
Ambrosetti, A.,
\textit{Critical points and nonlinear variational problems,}
M\'em.\ Soc.\ Math.\ Fr.\ S\'er.~2, \textbf{49} (1992), 1--139.
\bibitem{BF}
Bandle, C., Flucher, M.,
\textit{Harmonic Radius and Concentration of Energy; Hyperbolic Radius
and Liouville's Equations $\Delta U=e^U$ and $\Delta U=U^{\frac{n+2}{n-2}}$},
SIAM Review \textbf{38} (1996), 191--238.
\bibitem{Baraket2}
Baraket, S.,
\textit{Construction of singular limits for a strongly perturbed two-dimensional Dirichlet problem with
exponential nonlinearity,}
Bull.\ Sci.\ math.\ \textbf{123} (1999), 255--284.
\bibitem{BarPac}
Baraket, S., Pacard, F.,
\textit{Construction  of singular limits for a semiliear elliptic equation in dimension 2,}
Calc.\ Var.\ \textbf{6} n.~1 (1998), 1--38.
\bibitem{BP}  
Bartolucci, D., Pistoia, A. 
{\em  Existence and qualitative properties of concentrating solutions
for the sinh-Poisson equation.} 
IMA Journal of Applied Mathematics {\bf 72} (2007) 706--729.
\bibitem{BDP}
Bartsch, T., D'Aprile, T., Pistoia, A.,
{\em Multi-bubble nodal solutions for slightly subcritical elliptic problems in domains with symmetries,} 
Ann.\ Inst.\ H.~Poincar\'e Anal.\ Non Lin\'eaire \textbf{30} (2013), no.~6, 1027--1047.
\bibitem{BaMiPi}  
Bartsch, T., Micheletti, A.M., Pistoia, A., 
{\em  On the existence and profile of nodal solutions of elliptic equations
involving critical growth.} 
Calc.\ Var.\ \textbf{26} n.~3 (2006) 265--282.
\bibitem{BaPi}  
Bartsch, T., Pistoia, A., 
{\em  Critical Points of the $N-$vortex Hamiltonian in Bounded Planar Domains
and Steady State Solutions of the Incompressible Euler Equations.} 
SIAM J.\ Appl.\ Math.\ \textbf{75} (2015), n.~2, 726--744.
\bibitem{CaffarelliFriedman1985}  
Caffarelli, L., Friedman, A.,
{\em Convexity of solutions of semilinear elliptic equations,}
Duke.\ Math.\ J. \textbf{52} (1985), 431--457.
\bibitem{CLMP}  
Caglioti, E., Lions, P.L., Marchioro, C., Pulvirenti, M.,
{\em A special class of stationary flows for two-dimensional Euler equations:
A statistical mechanics description,}
Commun.\ Math.\ Phys.\ \textbf{174} (1995), 229--260.
\bibitem{EGP}
Esposito, P., Grossi, M., Pistoia, A.
{\em On the existence of blowing-up solutions for a mean field equation.}
Ann.\ Inst.\ H.~Poincar\'e Anal.\ Non Lin\'eaire  {\bf  22} (2005), 227--257.
\bibitem{EyinkSreenivasan}
Eyink, G., Sreenivasan, K.,
{\em Onsager and the theory of hydrodynamic turbulence,}
Rev.\ Modern Phys.\ \textbf{78} (2006), 87--135.
\bibitem{GT}
Grossi, M., Takahashi, F.,
{\em Nonexistence of multi-bubble solutions to some elliptic equations on convex domains,}
J.\ Funct.\ Anal.\ \textbf{259} (2010), 904--917.
\bibitem{Kiessling}
Kiessling, M.K.H.,
{\em Statistical mechanics of classical particles with logarithmic interactions,}
Comm.\ Pure Appl.\ Math.\ \textbf{46} (1993), 27--56.
\bibitem{Lin}
Lin, C.S.,
\textit{An expository survey on recent development of mean field equations,}
Discr.\ Cont.\ Dynamical Systems \textbf{19} n.~2 (2007), 217--247.
\bibitem{Moe}   
Moser, J., 
{\em  A sharp form of an inequality by N.~Trudinger.}
Indiana Univ.\ Math.\ J.\  {\bf 20} (1970/71), 1077--1092.
\bibitem{Neri}     
Neri, C., 
{\em Statistical Mechanics of the $N$-point vortex system with random intesities on a bounded domain.}
Ann.\ Inst.\ H.~Poincar\'e Anal.\ Non Lin\'eaire  {\bf  21} (2004), 381--399.
\bibitem{ORS}
Ohtsuka, H., Ricciardi, T., Suzuki, T.,
{\em Blow-up analysis for an elliptic equation describing stationary vortex flows with variable intensities in 2D-turbulence,}
J.\ Differential Equations \textbf{249} n.~6 (2010), 1436--1465.
\bibitem{On}
Onsager, L.,
{\em Statistical hydrodynamics,} 
Nuovo Cimento Suppl.\ n.~2 6 (9) (1949), 279--287.
\bibitem{RS}
Ricciardi, T., Suzuki, T.,
{\em Duality and best constant for a Trudinger--Moser inequality involving probability measures,}
J.\ of Eur.\ Math.\ Soc.\ (JEMS) \textbf{16} n.~7 (2014), 1327--1348.
\bibitem{RTZZ}
Ricciardi, T., Takahashi, R., Zecca G., Zhang X., 
{\em On the blow-up and existence of solutions for stationary vortex flows with variable intensities,}
in preparation.
\bibitem{RZ}
Ricciardi, T., Zecca, G.,
{\em Mass quantization and minimax solutions 
for Neri's mean field equation in 2D-turbulence},
arXiv:math.AP/1406.2925.
\bibitem{SawadaSuzuki}
Sawada, K., Suzuki, T.,
{\em Derivation of the equilibrium mean field equations of point vortex and vortex filament system,}
Theoret.\ Appl.\ Mech.\ Japan \textbf{56} (2008), 285--290.
\bibitem{Tru}     
Trudinger, N.S. 
{\em On imbeddings into Orlicz spaces and some applications.}
 J.\ Math.\ Mech.\  {\bf  17} (1967), 473--483.
\bibitem{Ye} 
Ye, D., 
\textit{Une remarque sur le comportement asymptotique des solutions de $-\Delta u=\lambda f(u)$,}
C.R.\ Acad.\ Sci.\ Paris \textbf{325} (1997), 1279--1282.
\end{thebibliography}
\end{document}